\documentclass[11pt]{article}

\input{packages}
\newcommand{\longname}{Optimal Shrinkage of Eigenvalues \\in the Spiked Covariance Model}

\theoremstyle{theorem} \newtheorem{thm}{Theorem}
\theoremstyle{definition} 
\theoremstyle{definition} \newtheorem{defn}{Definition}
\theoremstyle{theorem} \newtheorem{lemma}{Lemma}
\theoremstyle{theorem} \newtheorem{proposition}{Proposition}
\theoremstyle{definition}

\newcommand{\bitem}{\begin{itemize}}
\newcommand{\eitem}{\end{itemize}}
\newcommand{\goto}{\rightarrow}
\newcommand{\beq}{\begin{equation}}
\newcommand{\eeq}{\end{equation}}

\newcommand{\R}{\ensuremath{\mathbb{R}}}
\newcommand{\E}{\ensuremath{\mathbb{E}}}

\newcommand{\Nc}{\mathcal{N}}

\newcommand{\aslim}{\stackrel{a.s.}{\longrightarrow}}

\newcommand{\Pd}[3]{\ifthenelse{\equal{#3}{1}}{\frac{\partial #1}{\partial #2}}{\frac{\partial^{#3} #1}{\partial #2^{#3}}}}

\newcommand{\ObsOne}{{\sc [Obs. 1]}}
\newcommand{\ObsTwo}{{\sc [Obs. 2]}}
\newcommand{\ObsThree}{{\sc [Obs. 3]}}
\newcommand{\ObsFour}{{\sc [Obs. 4]}}
\newcommand{\Asy}{{\sc [Asy($\gamma$)]}}
\newcommand{\Spike}{{\sc [Spike($\ell_1,\ldots,\ell_r$)]}}
\newcommand{\SpikeSigma}{{\sc [Spike($\ell_1,\ldots,\ell_r|\sigma^2$)]}}
\newcommand{\OneSpike}{{\sc [Spike($\ell$)]}}
\newcommand{\argmin}{\mbox{argmin}}
\newcommand{\hf}{\tfrac{1}{2}}

\def\tr{{\rm tr}}

\def\E{{\mathbb E}}

\numberwithin{equation}{section}

\newtheorem*{theorem*}{Theorem}

\title{\longname}
\author{
David L. Donoho \footnotemark[1]
\and
Matan Gavish \footnotemark[2]
\and
Iain M. Johnstone \footnotemark[1]
}
\date{}

\begin{document}

\maketitle
\renewcommand{\thefootnote}{\fnsymbol{footnote}}
\footnotetext[1]{Department of Statistics, Stanford University}
\footnotetext[2]{School of Computer Science and Engineering, Hebrew University
of Jerusalem}
\renewcommand{\thefootnote}{\arabic{footnote}}

\begin{center}
  {\large \textsl{To the memory of Charles M. Stein, 1920-2016}}
\end{center}

\begin{abstract}
  We show that in a common high-dimensional covariance model, the choice of loss
function has a profound effect on optimal estimation. 

In an asymptotic framework 
based on the Spiked Covariance model and  use of orthogonally
invariant estimators, we show
that optimal estimation of the population covariance matrix 
boils down to design of an optimal shrinker 
$\eta$ that acts elementwise
on the sample eigenvalues.
Indeed, to each loss function there corresponds a unique admissible
eigenvalue 
shrinker $\eta^*$  dominating all other shrinkers. The shape of the
optimal shrinker  is 
determined by the choice of loss function and, crucially, by
inconsistency of both eigenvalues {\it and}  eigenvectors of the sample
covariance matrix.

Details of these phenomena and closed form formulas for the optimal eigenvalue 
shrinkers
are worked out for a menagerie of 26 loss functions
for covariance estimation 
found in the literature, including the Stein, Entropy, Divergence,
Fr\'{e}chet, Bhattacharya/Matusita, Frobenius Norm, Operator Norm,
Nuclear Norm and Condition Number losses.

\end{abstract}

\vspace{.1in}
\vspace{.1in}

\vspace{.1in}
{\bf Acknowledgements.}

We thank Amit Singer, Andrea Montanari, Sourav Chatterjee and Boaz Nadler for
helpful discussions. We also thank the anonymous referees for significantly
improving the manuscript through their helpful comments.  
This work was partially supported by NSF DMS-0906812  (ARRA).
MG was partially supported by a
William R. and Sara Hart Kimball Stanford Graduate Fellowship.

\vspace{.1in}
\newpage
\tableofcontents

%\pagebreak

%\input{keywords}

%\end{frontmatter}

\section{Introduction} \label{intro:sec}

Suppose we observe  
$p$-dimensional Gaussian vectors  $X_i \stackrel{i.i.d}{\sim} \Nc(0,\Sigma_p)$, $i=1,\dots,n$,
with $\Sigma = \Sigma_p$  the underlying  $p$-by-$p$  population
covariance matrix. 
To estimate $\Sigma$, 
we form the empirical (sample) covariance matrix $S =  S_{n,p} = n^{-1}
\sum_{i=1}^n X_i X'_i$; this is the maximum likelihood estimator.
Stein \cite{stein1956,stein1986} observed that the maximum
likelihood estimator $S$ ought to be improvable by eigenvalue shrinkage.

Write $S = V\Lambda V'$ for the eigendecomposition of $S$, where $V$
is orthogonal and the diagonal matrix
$\Lambda=\text{diag}(\lambda_1,\ldots,\lambda_p)$ 
contains the { empirical eigenvalues}.  Stein \cite{stein1986}
proposed to shrink the eigenvalues
by applying a specific nonlinear mapping 
$\varphi$
%univariate nonlinearity $\eta:[0,\infty)\to\mathbb{R}$ to
%  each eigenvalue of $S_{n,p}$, 
producing the estimate $\hat{\Sigma}_\varphi =
  V\mathbf{\varphi}(\Lambda)V'$, where $\varphi$ 
 maps the space of positive diagonal matrices onto itself.
% denotes the
%   application of $\eta$ 
%   entry-wise to the diagonal of $\Lambda$.  
In the ensuing half century,
  research on eigenvalue shrinkers has flourished, producing an extensive
  literature.  We can point here only to a fraction,
%of this literature, 
with  pointers organized into  early decades
  \cite{james1961estimation,EfronMorris1976,haff1979identity,haff1980empirical,berger1982estimation,haff1979estimation},
  the middle decades
  \cite{dey1985estimation,Sharma1985,SinhaGhosh,Kubokawa198969,krishnamoorthy1989improved,loh1991estimating,CJS:CJS91,pal1993estimating,yang1994estimation,gupta1995improved,LinPerlman1985},
  and the last decade
  \cite{daniels2001shrinkage,ledoit2004well,Sun2005455,huang2006covariance,karoui2008spectrum,ledoitPeche,ledoit2012nonlinear,fan2008high,chen2010shrinkage,won2012condition}.
%  Papers in this literature 
Such papers typically choose some loss function $L_p:S_p^+\times
  S_p^+\to[0,\infty)$, where $S_p^+$ is the space of positive semidefinite
    $p$-by-$p$ matrices, and develop a shrinker $\eta$ with ``favorable'' risk
    $\E\, L_p (\Sigma\,,\,\hat{\Sigma}_\eta(S))$. 
 
    In high dimensional problems, $p$ and $n$ are often of comparable
    magnitude. There, the maximum likelihood
    estimator is no longer a reasonable choice for covariance estimation and the
    need to shrink becomes acute.

In this paper, we consider a popular large $n$, large $p$  setting with $p$
comparable to $n$, and a set of assumptions about $\Sigma$ known as the
{\em Spiked Covariance Model} \cite{johnstone2001distribution}. 
We study a
variety of loss functions derived from or inspired by  the literature, and show
that to each ``reasonable'' nonlinearity $\eta$ there corresponds a well-defined asymptotic loss.

In the sibling problem of matrix denoising under a similar setting, it has been
shown that there exists a unique asymptotically admissible shrinker \cite{Shabalin2013, 2013arXiv1305.5870D}.
The same phenomenon is shown to exist here: for many different loss functions,
we show that there exists a  {\em unique optimal nonlinearity} $\eta^*$,
which we explicitly provide.  Perhaps surprisingly,  $\eta^*$ is the
only asymptotically admissible nonlinearity, namely, it offers equal or better
asymptotic loss than that of any other choice of $\eta$, across all possible 
Spiked Covariance models.

\subsection{Estimation in the Spiked Covariance Model} \label{assumptions:subsec}

Consider a sequence of covariance estimation problems, satisfying two 
basic assumptions.
\begin{description}
  \item[\Asy] The number of observations $n$ and
    the number of variables $p_n$ in the $n$-th problem 
follows the  proportional-growth limit $p_n/n \goto \gamma$, as $n \goto
\infty$, for a certain $0
< \gamma \leq 1$. 
\end{description} 

\noindent
Denote the population and sample covariances in the $n$-th problem by
$\Sigma = \Sigma_{p_n}$ and $S = S_{n,p_n}$ and
assume that the eigenvalues $\ell_i$ of $\Sigma_{p_n}$
satisfy:

\begin{description}
  \item[\Spike] 
%  \item[Spike($\ell_1,\ldots,\ell_r$).] 
The $r$ ``spikes'' $\ell_1 > \ldots > \ell_r \geq 1$ are fixed
independently of $n$ and $p_n$, and $\ell_{r+1} = \ldots = \ell_{p_n}
= 1$. 
\end{description} 

The spiked model exhibits three important phenomena, not seen in classical
fixed-$p$ asymptotics, that play an essential role in the construction of
optimal estimators. Drawing on results from
\cite{marvcenko1967distribution,baik2005phase,baik2006eigenvalues,paul2007asymptotics,Benaych-Georges2011,bai2008central},
we highlight:

\medskip
\textit{a. Eigenvalue spreading.} \ 
Consider model \Asy\ in the null case $\ell_1 = \ldots = \ell_r = 1.$ 
The empirical distribution of the sample eigenvalues 
$\lambda_{1n},\ldots,\lambda_{pn}$ 
converges as $n\to\infty$ to a
non-degenerate absolutely continuous distribution, the
Marcenko-Pastur or `quarter-circle' law \cite{marvcenko1967distribution}.
The distribution, or `bulk', is supported on a single interval, whose
limiting `bulk edges' are given by 
\begin{equation}
  \label{eq:lamplus}
  \lambda_{\pm}(\gamma) = (1\pm\sqrt{\gamma})^2.
\end{equation}

\medskip
\textit{b. Top eigenvalue bias.} \
Consider models \Asy\ and \Spike. 
For $i = 1, \ldots, r$, the leading sample eigenvalues satisfy
\begin{eqnarray}
  \label{eig_displacement:eq}
  \lambda_{in} & \aslim \,\lambda(\ell_i) \,,  
\end{eqnarray}
%\qquad i=1,\dots, r\,, \\
where the `biasing' function
\begin{eqnarray}
  \label{lambda_of_ell:eq}
  &
  \lambda(\ell) =   \ell + \gamma \ell/(\ell-1), 
  & \qquad \ell \geq \ell_+(\gamma)\,,
  \end{eqnarray}
\and $\lambda(\ell) \equiv (1 + \sqrt \gamma)^2 = \lambda_+(\gamma)$
for $\ell \leq \ell_+(\gamma)$, the 
\textit{Baik-Ben Arous-Pech\'{e} transition point} 
\begin{eqnarray}
\ell_+(\gamma) & = 1 + \sqrt{\gamma}\,.
\end{eqnarray}
% \begin{equation}
%   \label{eig_displacement:eq}
%   \lambda_{i,n} \goto_P \,\lambda(\ell_i) ,  \qquad i=1,\dots, r\,,
% \end{equation}
% where the `biasing' function
% \begin{equation}
%   \label{lambda_of_ell:eq}
%   \lambda(\ell) =   \ell + \gamma \ell/(\ell-1), 
%    \qquad \qquad \ell \geq \ell_+(\gamma),
% \end{equation}
% and $\lambda(\ell) \equiv (1 + \sqrt \gamma)^2 = \lambda_+(\gamma)$
% for $\ell \leq \ell_+(\gamma)$, the 
% \textit{Baik-Ben Arous-Pech\'{e} transition point}
% \begin{equation*}
%   \ell_+(\gamma) = 1 + \sqrt{\gamma}.
% \end{equation*}
Thus the empirical eigenvalues $\lambda_i$ are shifted upwards from
their theoretical counterparts $\ell_i$ by an asymptotically predictable
amount, of a size that exceeds $\gamma$ even for very large signal
strengths $\ell_i$.

\medskip
\textit{c. Top eigenvector inconsistency.} \
Again consider models \Asy\ and \Spike,  noting that $\ell_1 > \ldots >
\ell_r$ are distinct. The angles between the sample eigenvectors
$v_{1n},\ldots ,v_{pn},$ 
and the corresponding ``true'' population eigenvectors $u_{1n},\ldots,u_{pn}$ 
have non-zero limits:
\begin{equation}
  \label{eq:uv-incon}
  | \langle u_{in} , v_{jn} \rangle | \aslim
  \,\delta_{i,j} \cdot c(\ell_i)  \qquad 1 \leq i,j \leq r \,,
\end{equation}
where the cosine function is given by
\begin{equation}
\label{c_of_ell:eq}
    c(\ell) =  \sqrt{ \frac{1 -
    \gamma/(\ell-1)^2 }{1 + \gamma/(\ell-1)} }
      \qquad \ell \geq \ell_+(\gamma)\,,
\end{equation}
and $c(\ell)=0$ for $\ell\leq \ell_+(\gamma)$.
%\bigskip
\medskip

\textit{Loss functions and optimal estimation.} \ 
Now consider a class of estimators for the population covariance $\Sigma$, based
on {\em individual} shrinkage of the sample eigenvalues. Specifically,
\begin{equation}
  \label{eq:scalar-shrink}
  \hat \Sigma 
  = \hat \Sigma_\eta
  = \eta(\lambda_1) v_1 v_1' + \ldots + \eta(\lambda_p) v_p v_p',
\end{equation}
where $v_i$ is the sample eigenvector with sample eigenvalue $\lambda_i$ and
$\eta(\lambda)$ is a \textit{scalar nonlinearity}, 
$\eta : \R^+ \to [1,\infty)$, so that the \textit{same} function acts on each
sample eigenvalue. While this appears to be a significant restriction 
from Stein's use of vector functions $\varphi$ \cite{stein1986}, the discussion in 
Section \ref{optimality:sec} shows that nothing is lost in our setting
by the restriction to scalar shrinkers.
% \textbf{[Discussion below.]}

Consider a family of loss functions $L=\{L_p\}_{p=1}^\infty$ 
and a fixed nonlinearity $\eta:[0,\infty)\to\mathbb{R}$.
Define the asymptotic loss relative to $L$ of the shrinkage
estimator $\hat{\Sigma}_\eta$
%:S_{n,p_n}\mapsto V\eta(\Lambda) V'$
%in the spiked model satisfying assumption \Spike\, by
in model \Spike\ by
\begin{eqnarray} \label{asyloss-intro:eq}
  L_\infty(\ell_1,\ldots,\ell_r|\eta) = \lim_{n\to\infty}\,L_{p_n}\left(
  \Sigma_{p_n}\,,\,\hat{\Sigma}_\eta(S_{n,p_n})
  \right)\,,
\end{eqnarray}
assuming such limit exists.
If a nonlinearity $\eta^*$ satisfies 
\begin{equation}
  \label{eq:opt-shrink}
  L_\infty(\ell_1,\ldots,\ell_r|\eta^*) \leq 
  L_\infty(\ell_1,\ldots,\ell_r|\eta) 
\end{equation}
for any other nonlinearity $\eta$, any $r$ and any spikes
$\ell_1,\ldots,\ell_r$, and if for any $\eta$ the inequality is strict at some choice of
$\ell_1,\ldots,\ell_r$, then we
say that $\eta^*$ is the {\em unique asymptotically admissible} nonlinearity (nicknamed ``optimal'') for the loss sequence
$L$.

\begin{figure}[htp]
  \centering
  \includegraphics[width=5in]{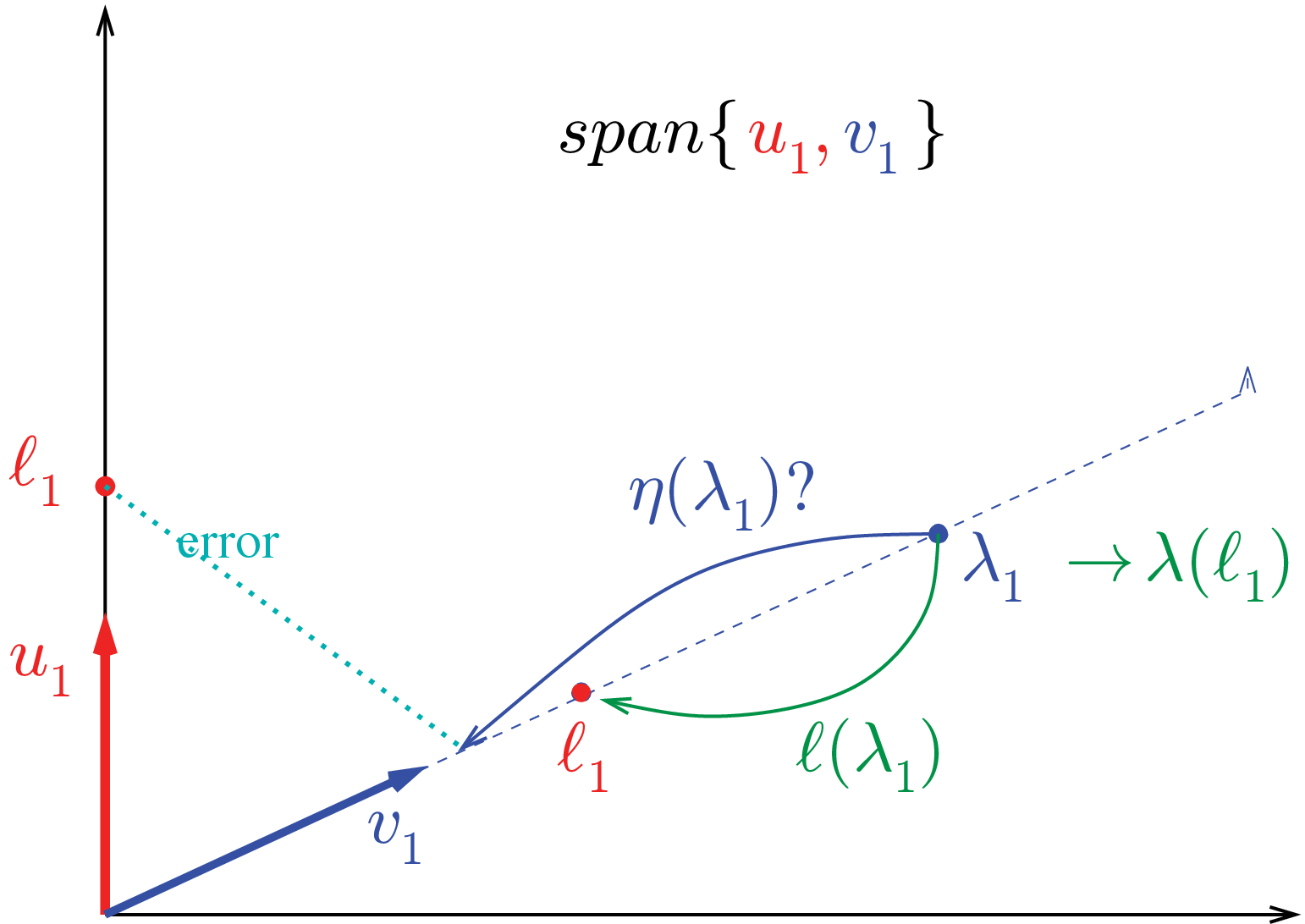}
  \caption{Shrinking empirical eigenvalue $\lambda_1$ to a value
    $\eta(\lambda_1)$ that is smaller than the inverse function
    $\ell(\lambda_1)$ may reduce the error of estimation.}
  \label{fig:shrinkage-pic}
\end{figure}

In constructing estimators, it is natural to expect that the effect
of the biasing function $\lambda(\ell)$ in (\ref{lambda_of_ell:eq})
might be undone simply by applying its inverse function $\ell(\lambda)$,
given by 
\begin{equation}
     \ell(\lambda) =
       \frac{(\lambda+1-\gamma) + \sqrt{(\lambda+1-\gamma)^2 - 4
            \lambda}}{2}   \qquad \qquad \lambda > \lambda_+(\gamma).
            \label{ell_of_lambda:eq}
\end{equation}
However, eigenvector inconsistency makes the situation more
complicated (and interesting!), as 
we illustrate using Figure \ref{fig:shrinkage-pic}.
Focus on the plane spanned by $u_1$, the top population
eigenvector, and by $v_1$, its sample counterpart.
We represent $\ell_1 u_1 u_1'$, the top rank one component of
$\Sigma$, by the vector $\ell_1 u_1$. 
The corresponding top rank one component of $S$ is $\lambda_1 v_1
v_1'$, represented by $\lambda_1 v_1$. If we apply the inverse
function (\ref{ell_of_lambda:eq}) to $\lambda_1$, we obtain 
$\ell(\lambda_1) v_1 v_1'$.
Since $v_1$ is not collinear with $u_1$, there is a non-vanishing
error $\ell(\lambda_1) v_1 v_1' - \ell_1 u_1 u_1'$ that remains, even
though $\ell(\lambda_1) - \ell_1 = O_p(n^{-1/2})$. 
As the picture suggests, it is quite possible that a different amount
of shrinkage, $\eta(\lambda_1) v_1 v_1'$ will lead to smaller error.
However, we will see that {\em the optimal choice of $\eta$ depends greatly
on the particular error measure $L_p(\Sigma, \hat \Sigma)$ that is chosen}.

To give the flavor of results to be developed systematically later, we
now look at four error measures in common use. 
The first three, based on the operator, Frobenius and nuclear norms,
use the singular values $\sigma_j$ of $\hat \Sigma - \Sigma$:
\begin{equation}
  \label{eq:4-losses}
\begin{split}
  L^O(\Sigma,\hat \Sigma) 
     & = \| \hat \Sigma - \Sigma \|_\infty \quad  = \  \max_i \sigma_i, \\
  L^F(\Sigma,\hat \Sigma) 
     & = \| \hat \Sigma - \Sigma \|_2  \quad \ =  \ \big( \sum_i \sigma_i^2
     \big)^{1/2},\\
  L^N(\Sigma,\hat \Sigma) 
     & = \| \hat \Sigma - \Sigma \|_1 \quad \ = \  \sum_i \sigma_i, \\
  L^{\rm St}(\Sigma,\hat \Sigma) 
     & = \tr (\Sigma^{-1} \hat \Sigma - I) - \log \det (\Sigma^{-1}
     \hat \Sigma).
\end{split}
\end{equation}
The fourth is Stein's loss, widely studied in covariance estimation 
\cite{stein1956,dey1985estimation,konno1991estimation}.

% We consider a class of estimators
% \begin{equation}
%   \label{eq:scalar-shrink}
%   \hat \Sigma 
%   = \hat \Sigma_\eta
%   = \eta(\lambda_1) v_1 v_1' + \ldots + \eta(\lambda_p) v_p v_p',
% \end{equation}
% where $\eta(\lambda)$ is a \textit{scalar nonlinearity}, 
% $\eta : \R^+ \to [1,\infty)$, the \textit{same} fucntion on all
% eigenvalues. \textbf{[Discussion below.]}
For convenience, we begin with the single spike model
\textbf{Spike($\ell$)}, so that $\Sigma = \Sigma_\ell = I + (\ell - 1)
u_1 u_1'$.
When $\eta$ is continuous, the losses have a deterministic asymptotic
limit $L_\infty(\ell|\eta)$ defined in (\ref{asyloss-intro:eq}).
% , which we call the asymptotic loss, or asy-loss:
% \begin{equation*}
%   L_\infty(\ell|\eta)
%   = \text{a.s.lim} \, L_p(\Sigma_\ell,\hat \Sigma_\eta).
% \end{equation*}

For many losses, including (\ref{eq:4-losses}), this
deterministic  limiting loss has a simple form, and we can evaluate, often
analytically, the optimal shrinkage function, namely the shrinkage function
satisfying 
 \eqref{eq:opt-shrink}.
% \begin{equation*}
%   \eta_*(\ell) = \argmin_\eta L_\infty(\ell,\eta).
% \end{equation*}
For example, writing $\eta^*(\lambda) = \eta_*(\ell(\lambda))$, for the four
popular 
loss
functions \eqref{eq:4-losses}  we find that on
 $\ell > 1 + \sqrt \gamma$ the corresponding four optimal shrinkers are
% Lemmas
% \ref{lem:operatorNormNonlinearity}--\ref{lem:steinLossNonlinearity} below,
\begin{alignat}{2}
  \label{eq:four-solns}
    \eta_*^O(\ell) & = \ell  
    & \eta_*^F(\ell) & = \ell c^2 + s^2 \\
    \eta_*^N(\ell) & = \max( 1 + (\ell -1)(1 -2 s^2), 1) \qquad
    & \eta_*^{\rm St}(\ell) 
       & = \ell/(c^2 + \ell s^2)\,, \notag
\end{alignat}
% \begin{equation}
%   \label{eq:four-solns}
%   \begin{split}
%     \eta_*^O(\ell) & = \ell,  \\
%     \eta_*^F(\ell) & = \ell c^2 + s^2, \\
%     \eta_*^N(\ell) & = \max( 1 + (\ell -1)(1 -2 s^2), 1), \\
%     \eta_*^{\rm St}(\ell) 
%        & = \ell/(c^2 + \ell s^2),
%   \end{split}
% \end{equation}
where $s^2 = 1 - c^2$.
 Figure \ref{fig:opt-shrinkers} shows these four optimal shrinkers as a function
 of the sample eigenvalue $\lambda$. 
These are just four examples; The full list of optimal shrinkers we discover in this paper appears in Table
\ref{Table:formulas} below.
In all cases, $\eta_*(\ell) \equiv 1$ for $\ell \leq 1 + \sqrt \gamma$.
Figure \ref{fig:OptimalShrinkers} in Section \ref{sec:OptShrink}
below shows all the full list of optimal shrinkers when
$\gamma=1$.

\begin{figure}[h]
  \centering
  \includegraphics[width=5in,clip]{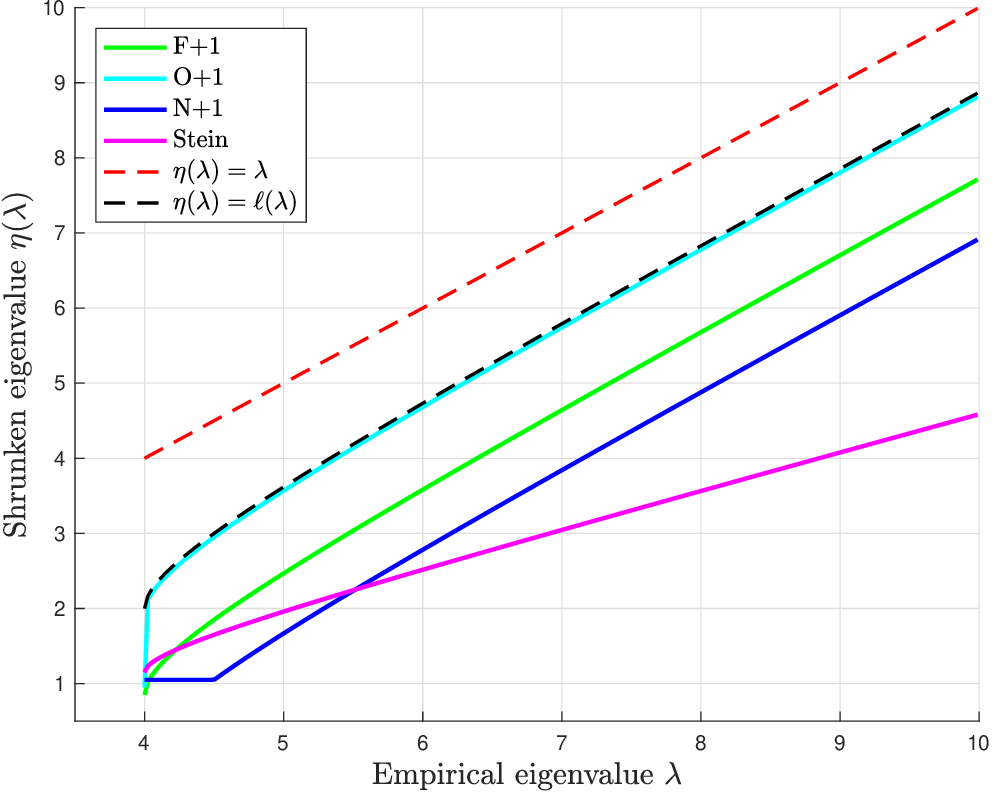}
  \caption{Vertical axis: optimal shrinkers $\eta_*$ from
    (\ref{eq:four-solns}), shown as functions $\eta_*(\ell(\lambda))$
   of the empirical eigenvalue $\lambda$, horizontal axis.
   Here $\gamma = \lim p_n/n = 1$, so $\lambda_+(\gamma) = 4$. (Color online.)
} 
  \label{fig:opt-shrinkers}
\end{figure}

The main conclusion is that the \textit{optimal shrinkage function depends
strongly on the loss function} chosen. 
The operator norm shrinker $\eta_*^O$ simply inverts the biasing
function $\lambda(\ell)$, while the other functions shrink by much
larger, and very different, amounts, with $\eta_*^{\rm St}$
typically shrinking most.
There are also important qualitative differences in the optimal
shrinkers: $\eta_*^O$ is discontinuous at the bulk edge
$\lambda = \lambda_+(\gamma)$. 
The others are continuous, but $\eta_*^N$ has the additional feature
that it shrinks a \textit{neighborhood} of the bulk to $1$.

\smallskip
\textit{Remark.} \ 
The optimal shrinker also depends on $\gamma$, so we might write
$\eta^*(\lambda, \gamma)$.
In model \Asy, one can use the same $\gamma$ for each problem size
$n$. 
Alternatively, in the $n$-th problem, one might use $\gamma_n =
p_n/n$. 
The former choice is simpler, as $\eta^*$ can be regarded as a
univariate function of $\lambda$, and so we make it in Sections 
\ref{intro:sec}--\ref{sec:OptShrink}. 
The latter choice is preferable technically, and perhaps also in
practice, when one has $p$ and $n$, but not $\gamma$. It does, however, 
require us to treat $\eta(\lambda, c)$ as a bivariate
function -- see Section \ref{multspike:sec}.

%\newpage

\subsection{Some key observations}
\label{sec:some-key-observ}

The sections to follow construct a framework for evaluating 
and optimizing
the asymptotic loss \eqref{asyloss-intro:eq}.
We highlight here some observations that will play an important role.
Beforehand, let us introduce a useful 
%An important role is played by a technical 
modification of
(\ref{eq:scalar-shrink}) to a \textit{rank-aware} shrinkage rule:
\begin{equation}
  \label{eq:rank-aware}
  \hat \Sigma_{\eta,r}
   = \sum_{i=1}^r \eta(\lambda_i) v_i v_i' + \sum_{i=r+1}^p  v_i v_i',
   % = \eta(\lambda_1) v_1 v_1' + \ldots \eta(\lambda_r) v_r v_r' + 
   %     v_{r+1} v_{r+1}' + \ldots + v_p v_p',
\end{equation}
where the dimension $r$ of the spiked model is taken as known.
While our main results concern estimators $\hat \Sigma_\eta$
that naturally do not require $r$ to be known in advance, 
 it will be easier conceptually and technically 
to analyze rank-aware shrinkage rules as a preliminary step.

\smallskip
\noindent
\ObsOne\,\, \textit{Simultaneous block diagonalization.} \ (Lemmas
\ref{jointBlock-perlim} and \ref{lem:rank-r-decomp}). 
There exists a (random) basis $W$ such that
\begin{equation*}
  \begin{split}
    W' \Sigma W & = (\oplus_i A_i) \oplus I_{p-2r} \\
    W' \hat \Sigma_{\eta,r} W & = (\oplus_i B_i) \oplus I_{p-2r},
  \end{split}
\end{equation*}
where $A_i$ and $B_i$ are square blocks of equal size $d_i$, and $\sum
d_i=2r$. (Here and below, $A\oplus B$ denotes a
block-diagonal matrix with blocks $A$ and $B$).

\smallskip
\noindent
\ObsTwo\,\, \textit{Decomposable loss functions.} \ 
The loss functions \eqref{eq:4-losses} and many others studied below
satisfy
\begin{equation*}
  L_p(\Sigma,\hat{\Sigma}_{\eta,r} )=\sum_i L_{d_i}(A_i,B_i)
\end{equation*}
or the corresponding equality with {\em sum} replaced by {\em max}.

\smallskip
\noindent
\ObsThree\,\, \textit{Asymptotic deterministic loss.} \ 
(Lemmas \ref{lem:asymp-loss} and \ref{lem:rank-aware-lim}).
For rank-aware
estimators, when $\eta$ and $L$ are suitably continuous, almost surely
\begin{equation*}
  L_\infty(\ell_1, \ldots, \ell_r|\eta)
   = \lim_{p \to \infty} L_p(\Sigma,\hat \Sigma_{\eta,r}).
%   \aseq \lim_{p \to \infty} L_p(\Sigma,\hat \Sigma_{\eta,r}).
\end{equation*}

\smallskip
\noindent
\ObsFour\,\, \textit{Asymptotic equivalence of losses.} \ 
(Proposition \ref{prop:removing-rank-aware}).
Conclusions derived for rank-aware estimators
(\ref{eq:rank-aware}) carry over to the original estimators
(\ref{eq:scalar-shrink}) because, under suitable conditions
\begin{equation*}
    L_p(\Sigma, \hat\Sigma_{\eta})
   - L_p(\Sigma, \hat\Sigma_{\eta,r}) \to_P 0.
\end{equation*}
This relies on the fact that in the \Spike\ model, the sample
noise eigenvalues $\lambda_{in}, i\geq r + 1$ ``stick to the bulk'' in
an appropriate sense.

\subsection{Organization of the paper} \
\label{sec:organization-paper}
%\textit{Organization of the paper.} \ 
%\textnormal{[REDO LAST.]} \ 
For simplicity of exposition, we assume a
single spike, $r=1$, in the first half of the paper. 
\ObsOne, \ObsTwo \,and \ObsThree \, are developed respectively in 
Sections \ref{sec:block}, \ref{sec:DecomposableLoss} and
\ref{sec:asyloss}, arriving at an explicit
formula for the asymptotic loss of a shrinker.
%  \ObsOne\, is developed in
% Section \ref{sec:block}. {Section \ref{sec:DecomposableLoss}  fleshes out
% \ObsTwo\, and introduces our list of 26 decomposable matrix loss functions.
% %Section \ref{sec:spiked} includes background on the spiked  model.  
% Section \ref{sec:asyloss} discusses \ObsThree\, above and derives an explicit
% formula
% for the asymptotic loss of a shrinker. 
Section \ref{sec:exampl-decomp-loss} 
illustrates the assumptions with our list of 26
decomposable matrix loss functions.
In Section \ref{sec:OptShrink} we
use the formula to characterize the asymptotically unique admissible
nonlinearity for any decomposable loss, provide an algorithm for computing the
optimal nonlinearity, and provide analytical formulas for many of the
26 losses.   Section
\ref{multspike:sec}  extends the results to the general case where $r>1$
spikes are present.  
We develop \ObsFour\,,  remove the rank-aware assumption and explore some new phenomena
that arise in cases where the optimal shrinker turns out to be discontinuous. 
In Section \ref{optimality:sec} we show, at least for
Frobenius and Stein losses, that our optimal univariate shrinkage
estimator, which applies the same scalar function to each  sample
eigenvalue, in fact asymptotically matches the performance of the best
{\em orthogonally-equivariant} covariance
estimator under assumption \Spike.
Section \ref{sigma:sec} extends 
% we show how these optimal nonlinearities can be used to
% construct asymptotically
% unique admissible nonlinearities in 
to the more general spiked model with
$\Sigma_p =
  diag(\ell_1,\ldots,\ell_r,\sigma^2,\ldots,\sigma^2)$ for $\sigma>0$
  known or unknown.
Section \ref{discussion:sec}  discusses our results in light of the
 high-dimensional covariance estimation work of 
 El Karoui \cite{karoui2008spectrum} and Ledoit and Wolf \cite{ledoit2012nonlinear}.
  Some proofs and calculations 
are deferred to the supplementary article \cite{SI}, 
where we also evaluate and document the strong signal
(large-$\ell$) 
asymptotics of the optimal shrinkage estimators, and the asymptotic percent
improvement over naive hard thresholding of the sample covariance eigenvalues.
% , namely, their behavior and their
% performance when the signal $\ell$ is very strong. 
Additional technical details and software are provided in the
Code Supplement available online as a permanent URL from the Stanford Digital Repository \cite{SDR}.

%\newpage

\section{Simultaneous Block-Diagonalization} \label{sec:block}

We first develop \ObsOne\, in the simplest case, $r=1$, assumping a rank-aware
shrinker.
In general, the estimator $\hat \Sigma_\eta$ and estimand $\Sigma$ are
not simultaneously diagonalizable. 
However, in the particular case that both are rank-one perturbations
of the identity, we will see that simultaneous block diagonalization is
possible. 

Some notation is needed.
We denote the eigenvalues and eigenvectors of the spectral
decompostion $S_{n,p_n} = V\Lambda V'$ by
% We have written $S_{n,p} = V\Lambda V'$ for
% a spectral decomposition of $S_{n,p}$.
% To indicate the eigen-values and -vectors explicitly, we use the
% notation
\begin{equation*}
    spec(S_{n,p_n})
    = [(\lambda_{1n}, \ldots, \lambda_{pn}), (v_{1n},\ldots,v_{pn})]\,.
\end{equation*}
Whenever possible, we supress the index $n$ and write e.g. $S$, $\lambda_i$ and $v_i$
instead. Similarly, we often write $\Sigma_{p}$ or even $\Sigma$ for
$\Sigma_{p_n}$.

\begin{lemma} \label{jointBlock-perlim}
Let $\Sigma$ and $\hat{\Sigma}$ be (fixed, nonrandom) $p$-by-$p$ 
symmetric positive definite matrices with
\begin{align}
  \label{eq:specSigma}
  spec(\Sigma) 
    & = [(\ell, 1, \ldots, 1), (u_1,\ldots,u_p)] \\
  \label{eq:specSigmaHat}
  spec(\hat \Sigma)
    & = [(\eta, 1, \ldots, 1), (v_1,\ldots,v_p)]. \\
%\end{align}
\intertext{Let $c = \langle u_{1}, v_{1} \rangle$ and  $s = \sqrt{1 - c^2}$.
Then there exists an orthogonal matrix $W$, which depends on $\Sigma$ and
  $\hat{\Sigma}$, such that}
%\begin{align}
  W' \Sigma W 
     & = A(\ell) \oplus I_{p-2},  \label{eq:sig} \\    
  W' \hat \Sigma W 
     & = B(\eta, c) \oplus I_{p-2}, \label{eq:sighat}    
\end{align}
where the fundamental $2 \times 2$ matrices $A$ and $B$ are
  given by
% \intertext{and the fundamental $2 \times 2$ matrices $A$ and $B$ are
%   given by}
\begin{align}
  A(\ell) 
     & =
     \begin{bmatrix}
       \ell & 0 \\
       0    & 1 
     \end{bmatrix},
     \qquad \quad
  B(\eta, c) 
      = I_2 + (\eta-1) 
      \begin{bmatrix} c \\ s  \end{bmatrix}
   \begin{matrix}
      \begin{bmatrix} c & s   \end{bmatrix}\\\mbox{}
   \end{matrix}\,\,.
         \label{eq:ABforms}
      % \begin{bmatrix} c \\ s  \end{bmatrix}
      % \begin{bmatrix} c & s   \end{bmatrix}.  \label{eq:ABforms}
\end{align}
\end{lemma}
\begin{proof}
  Let $\Delta = \text{diag}(\eta,1,\ldots,1) = I + (\eta-1)e_1 e_1'$,
  where $e_1$ denotes the unit vector in the first co-ordinate
  direction.
It is evident that
\begin{equation} 
  \label{eq:rankone}
  \Sigma = I + (\ell -1) u_1 u_1', \qquad 
  \hat \Sigma = I + (\eta -1) v_1 v_1'.
\end{equation}
It is natural, then, to work in the ``common'' basis of $u_1$ and
$v_1$. We apply one step of Gram-Schmidt if we can, setting
\begin{equation*}
  z =
  \begin{cases}
    (v_1 - c u_1)/s & \quad \text{if } s \neq 0 \\
    u_p & \quad \text{if } s = 0.
  \end{cases}
\end{equation*}
In the second--exceptional--case, $v_1 = \pm u_1$, so we pick a
convenient vector orthogonal to $u_1$.
In either case, the columns of the $p \times 2$ matrix $W_2 = [u_1 \
z]$ are orthonormal and their span contains both $u_1$ and $v_1$. 
Now fill out $W_2$ to an orthogonal matrix 
$W = [W_2 \ \, W_2^\perp]$. 
Observe now that if $y$ lies in the column span of $W_2$ and $\alpha$
is a scalar, then necessarily
\begin{equation*}
  W'(I_p + \alpha y y') W 
   = (I_2 + \alpha \check y \check y) \oplus I_{p-2},
  \qquad \check y = W_2'y.
\end{equation*}
The expressions \eqref{eq:sig} -- \eqref{eq:ABforms} now follow from 
the rank one perturbation forms \eqref{eq:rankone} along with
\begin{equation*}
  W_2' u_1 =
  \begin{bmatrix}
    u_1'u_1 \\
    z' u_1
  \end{bmatrix} 
  =
  \begin{bmatrix}
    1 \\ 0 
  \end{bmatrix}, \quad \text{and} \quad 
    W_2' v_1 =
  \begin{bmatrix}
    u_1'v_1 \\
    z' v_1
  \end{bmatrix} 
  =
  \begin{bmatrix}
    c \\ s 
  \end{bmatrix}.   \qedhere
\end{equation*}

\end{proof}

\section{Decomposable Loss Functions} \label{sec:DecomposableLoss}

Here and below, by {\em loss function} $L_p$ we mean a function of two
$p$-by-$p$ positive semidefinite matrix arguments obeying $L_p\geq 0$, with
$L_p(A,B)=0$ if and only if $A=B$. A {\em loss family} is a sequence 
$L=\left\{ L_p \right\}_{p=1}^\infty$, one for each matrix size $p$. We often
 write loss function and refer to the entire family.
\ObsTwo\, calls out a large class of loss
functions which naturally exploit the simultaneously block-diagonalizability
property of Lemma \ref{jointBlock-perlim};
%\ref{cor-jointBlock}; 
we now develop this observation.

\begin{defn} \label{def-OrthInvar}  {\bf Orthogonal Invariance.}
  We say the loss function $L_p(A,B)$ is {\em orthogonally invariant} if
  for each orthogonal $p$-by-$p$ matrix $O$,
\[L_p(A,B) = L_p(OAO',OBO').\]
\end{defn}

For given $p$ and a given sequence of block sizes $\{d_i\}$ such that
$\sum_i d_i = p$,  
consider block-diagonal matrix decompositions of $p$ by $p$ matrices $A$ and $B$ 
into blocks $A^i$ and $B^i$ of size $d_i$:
\begin{equation}
  \label{eq:decomps}
  A = \oplus_i A^i \qquad \qquad
  B = \oplus_i B^i.
\end{equation}
% where $A^i$ and $B^i$ and $d_i$-by-$d_i$ matrices.

\begin{defn} \label{def-SumDecomp} {\bf Sum-Decomposability and
  Max-Decomposability.}
We say the loss function $L_p(A,B)$ 
is {\em sum-decomposable} 
if for all decompositions \eqref{eq:decomps},
\[
 L_p(A,B) = \sum_i L_{d_i}(A^i,B^i) \,.
\]
We say that it is {\em max-decomposable} if
if for all decompositions \eqref{eq:decomps},
\[
 L_p(A,B) = \max_i  L_{d_i}(A^i,B^i)\,.
\]
\end{defn}

Clearly, such loss functions can exploit the simultaneous block diagonalization
of Lemma \ref{jointBlock-perlim}. Indeed,
\begin{lemma}  \label{lem:asympDecomp}
{\bf Reduction to Two-Dimensional Problem.}
Consider an orthogonally invariant loss function, $L_p$, 
which is sum- or max-decomposable. 
Suppose that $\Sigma$ and $\hat \Sigma$ satisfy (\ref{eq:specSigma})
and (\ref{eq:specSigmaHat}) respectively. Then
% Consider the single spike model \textbf{Spike($\ell$)} and a rank
% aware estimator $\hat \Sigma_{\eta,1}$, (\ref{eq:rank-aware}). Then
\[
  L_p(\Sigma,\hat{\Sigma})
 = L_2(A(\ell),B(\eta,c)).
\]
\end{lemma}
\begin{proof}
Lemma \ref{jointBlock-perlim} provides a change of basis 
% In the single spike model, $\Sigma$ is given by (\ref{eq:specSigma}),
% and so we can apply Lemma \ref{jointBlock-perlim} taking $\eta$ in
% (\ref{eq:specSigmaHat}) to be $\eta(\lambda_1)$. Thus there is a
% (now random) change of basis 
$W$ yielding decompositions
(\ref{eq:sig}) and (\ref{eq:sighat}). 
From the invariance and decomposability hypotheses,
%Hence
\begin{align*}
  L_p(\Sigma, \hat \Sigma) 
  & = L_p( W' \Sigma W, W' \hat \Sigma W) \\
  & = L_p(A(\ell)\oplus I_{p-2}, B(\eta),c)\oplus
I_{p-2}) \\
  & = L_2(A(\ell),B(\eta,c)).  \qquad \qquad  \qedhere
\end{align*} 
% \begin{align*}
%   L_p(\Sigma, \hat \Sigma_{\eta,1}) 
%   & = L_p( W' \Sigma W, W' \hat \Sigma_{\eta,1} W) \\
%   & = L_p(A(\ell_1)\oplus I_{p-2}, B(\eta(\lambda_1),c_1)\oplus
% I_{p-2}) \\
%   & = L_2(A(\ell_1),B(\eta(\lambda_1),c_1)),
% \end{align*}
%after applying the invariance and decomposability hypotheses.
\end{proof}

% Focusing again on the single spike case ($r=1$), when the representation
% \eqref{rep:eq} holds, we have for any orthogonally
% invariant, sum-decomposable or max-decomposable loss function $L_p$ that
% \begin{eqnarray*}
%   L_p\left(\Sigma_p,\hat{\Sigma}_\eta(S_{n,p})\right) &=&  
%   L_p\left(A(\ell_1)\oplus I_{p-2}\,,\, B(\eta(\lambda_1),c_1,s_1)\oplus
% I_{p-2}\right) \\ &=&  L_2\left(A(\ell_1)\,,\,B(\eta(\lambda_1),c_1,s_1)\right)\,.
% \end{eqnarray*}
% In other words, evaluating $L_p$ reduces to evaluating $L_2$ on specific
% matrices $A$ and $B$, and we may ignore the last $p-2$ coordinates completely.
% We have proved:
% \begin{lemma}  \label{lem:asympDecomp1}
% {\bf Reduction to Two-Dimensional Problem.}
% Consider an orthogonally invariant loss function, $L_p$, 
% which is sum- or max-decomposable. Assume 
% that the sample covariance $S_{n,p}$ and nonlinearity $\eta$ are such that 
%  Lemma \ref{cor-jointBlock} holds, so that $\Sigma_p$ and
% $\hat{\Sigma}_\eta(S_{n,p})$ are simultaneously block-diagonalizable.
% Then
% \[
%   L_p\left(\Sigma_p,\hat{\Sigma}_\eta(S_{n,p})\right) 
%  = L_2\left(A(\ell_1)\,,\,B(\eta(\lambda_1),c_1,s_1)\right)\,.
% \]
% \end{lemma}

%\newpage
\section{Asymptotic Loss in the Spiked Covariance Model} \label{sec:asyloss}

Consider the spiked model with a single spike, $r=1$, namely, make
assumptions \Asy\, and \OneSpike. 
The principal $2 \times 2$ block estimator occurring in Lemmas
\ref{jointBlock-perlim} and \ref{lem:asympDecomp} is
$B(\eta(\lambda_{1n}),c_{1n})$ 
where $\lambda_{1n}$ is the largest eigenvalue of $S_n$ and 
$c_{1n} = \langle u_{1n}, v_{1n} \rangle$. 

If $\eta$ is continuous, then the convergence results
(\ref{eig_displacement:eq}) and (\ref{eq:uv-incon})
imply that the principal block converges as $n\to\infty$. Specifically,
\begin{equation}
  \label{eq:p-block-cge}
  B(\eta(\lambda_{1n}),c_{1n})
  \aslim
  B(\eta(\lambda(\ell)), c(\ell)) 
   =: B(\ell, \eta),
\end{equation}
say, with the convergence occurring in all norms on $2 \times 2$ matrices.

In accord with \ObsThree, we now show that the asymptotic loss 
\eqref{asyloss-intro:eq}
is a deterministic, explicit function of the population spike $\ell$.
For now, we will continue to assume that the shrinker $\eta$ is rank-aware.
Alternatively, we can make a different simplifying assumption on $\eta$, which
will be useful in what follows:

\begin{defn}  \label{def:bulk-shrinkers}
  We say that a scalar function $\eta: [0, \infty) \to [1, \infty)$ is
a \textit{bulk shrinker} if $\eta(\lambda) = 1$ when $\lambda \leq
\lambda_+(\gamma)$, and 
a \textit{neighborhood bulk shrinker} if for some $\epsilon > 0$,
$\eta(\lambda) = 1$ whenever $\lambda \leq
\lambda_+(\gamma) + \epsilon.$
\end{defn}

The neighborhood bulk shrinker condition on $\eta$
is rather strong, but does hold
for $\eta_*^N$ in (\ref{eq:four-solns}), for example.
(Note that our definitions ignore the lower bulk edge
$\lambda_-(\gamma)$, which is of less interest in the spiked model.)

% For $\eps > 0$,  consider the event
% %
% \[
%    \Omega_{n,\eps} = \{ \lambda_{1,n} > \lambda_+(\gamma) + \eps ,
%    \lambda_{2,n} < \lambda_+(\gamma) + \eps \}.
%    % \cap 
%    % \left\{ u_{1,n} \not\perp v_{p_n,n} \right\}\,.
% \]
% %
% %Lemcontrol the probability of $\Omega_{n,\eps}$, we need the following lemma.

% Under the single-spiked model  and the assumption $\ell_1 > \ell_+(\gamma)$ 
% it now follows from Theorem \ref{thm:spikedModel} 
% %and Lemma \ref{notperp} 
% that for sufficiently small $\eps > 0$ we have
% $
% P(\Omega_{n,\eps} ) \goto 1 \mbox{ as } n \goto \infty
% $.

% \begin{defn} \label{def:collapseBulk} {\bf Collapsing the bulk.}
% We say that a scalar nonlinearity $\eta:[0,\infty)\to[0,\infty)$ {\em collapses the bulk to 1} if 
% $\eta(\lambda)=1$ whenever $ \lambda\in(\lambda_-(\gamma),\lambda_+(\gamma) )$. 
%  We further say that a scalar nonlinearity $\eta:[0,\infty)\to[0,\infty)$ {\em collapses the vicinity  of the bulk to 1}
% if,  for some $\eps > 0$, we have
% $\eta(\lambda)=1$ whenever $ \lambda\in(\lambda_-(\gamma) - \eps \,,\,\lambda_+(\gamma) + \eps)$. 
% \end{defn}

%\newpage
\begin{lemma}
  \label{lem:asymp-loss}
  {\bf A Formula for the Asymptotic Loss.} 
Adopt models \Asy\, and \OneSpike\, with
$\ell > \ell_+(\gamma)$. 
Suppose (a) that the family $L=\left\{ L_p \right\}$ of loss functions is
orthogonally invariant and sum- or max- decomposable, and that
$B \mapsto L_2(A, B)$ is continuous. 
Let $\hat\Sigma_\eta=\hat{\Sigma}_\eta(S_{n,p_n})$ be given by 
(\ref{eq:scalar-shrink}), and let $\hat{\Sigma}_{\eta,1}$
be the corresponding rank-aware shrinkage rule (\ref{eq:rank-aware}) for $r=1$.
Suppose the scalar nonlinearity $\eta$ is
continuous on $\left( \lambda_+(\gamma),\infty \right).$
%and collapses the vicinity of the  bulk to 1. 
Then
\begin{equation} \label{eq:asy-formula}
  L_{p_n}(\Sigma_{p_n}, \hat \Sigma_{\eta,1}) 
  \aslim 
  L_2(A(\ell), B(\ell, \eta))\,,
\end{equation}
 Furthermore,
if (b) $\eta$
is a neighborhood bulk shrinker, then 
$ L_{p_n}(\Sigma_{p_n}, \hat \Sigma_{\eta}) $ also has this limit a.s.

% where we make the abbreviation
% \begin{equation}
%   \label{eq:1}
%   B(\ell,\eta) = B(\eta(\lambda(\ell)),c(\ell))).
% \end{equation}
\end{lemma}
\noindent Each of the 26 losses considered in this paper satisfies
conditions (a).
\begin{proof}
In the rank-aware case $\hat{\Sigma}_\eta = \hat{\Sigma}_{\eta,1}$
satisfies 
\begin{equation*}
  \text{spec}(\hat \Sigma_\eta) = [(\eta(\lambda_{1n}),1, \ldots, 1),
  (v_{1n}, \ldots, v_{pn})],
\end{equation*}
Lemma \ref{lem:asympDecomp} implies that 
\begin{align*}
  L_p(\Sigma, \hat \Sigma_\eta) 
     = L_2(A(\ell), B(\eta(\lambda_{1n}),c_{1n}))  
     \aslim L_2(A(\ell), B(\ell, \eta))\,,
\end{align*}
where the limit on the right hand side follows from 
convergence 
(\ref{eq:p-block-cge}) and the
assumed continuity of $L_2$.
  
Now assume that $\eta$ is a neighborhood bulk shrinker. 
From (\ref{eig_displacement:eq}) we know that $\lambda_{1n} \aslim
\lambda(\ell)$  
From eigenvalue interlacing (see \eqref{eq:interlace} below) we have
$\lambda_{2n} \leq \mu_{1n}$, 
where $\mu_{1n}$ is the largest eigenvalue of a white Wishart matrix
$W_{p_n -1}(n,I)$, and satisfies 
$\mu_{1n} \aslim \lambda_+$, from \cite{gema80}.
Let $\epsilon > 0$  be small enough that $\lambda_+ + \epsilon <
\lambda(\ell)$ and also
lies in the neighborhood shrunk to $1$ by $\eta$.
Hence, there exists 
% We have both $\lambda_{1n} \to \ell$ and $\limsup \lambda_{2n} \leq
% \lambda_+$ almost surely (by interlacing!).
% Consequently, there exists 
a random variable $\hat n$ such that
almost surely,
$\lambda_{2n} < \lambda_+ + \epsilon < \lambda_{1n}$
 for all $n > \hat n$. 
For such $n$, the first display above of this proof applies and we
then obtain the second display as before.
% On event $\Omega_{n,\epsilon}$, 
% since $\eta$ collapses the vicinity of the bulk,
% %our assumption on $\eta$ implies
% \begin{equation*}
%   \text{spec}(\hat \Sigma_\eta) = [(\eta(\lambda_{1n}),1, \ldots, 1),
%   (v_{1n}, \ldots, v_{pn})],
% \end{equation*}
% an equality that also holds by definition for the rank-aware case.
% Now, Lemma \ref{lem:asympDecomp} implies that 
% \begin{align*}
%   L_p(\Sigma, \hat \Sigma_\eta) 
%      = L_2(A(\ell), B(\eta(\lambda_{1n}),c_{1n}))  
%      \aslim L_2(A(\ell), B(\ell, \eta))\,,
% \end{align*}
% where the limit on the right hand side follows from 
% convergence (\ref{eq:p-block}) and the
% assumed continuity of $L_2$.
% We use Lemma \ref{jointBlock-perlim}, invariance and decomposability
% to get, on $\Omega_{n,\epsilon}$, 
% %may thus apply Lemma REF and the assumptions on $L$ to derive
% \begin{align*}
%   L_p(\Sigma, \hat \Sigma_\eta) 
%     & = L_p(W'\Sigma W, W' \hat \Sigma_\eta W) \\
%     & = L_p(A(\ell)\oplus I_{p-2}, B(\eta(\lambda_{1n}),c_{1n}) \oplus
%     I_{p-2}) \\
%     & = L_2(A(\ell), B(\eta(\lambda_{1n}),c_{1n}))  \\
%     & \stackrel{P}{\to} L_2(A(\ell), B(\eta(\lambda(\ell_1)),c(\ell_1))),
% \end{align*}
% using the convergence $\lambda_{1n} \stackrel{P}{\to}
% \lambda(\ell_1)$ and 
% $c_{1n} \stackrel{P}{\to} c(\ell_1)$ from Theorem
% \ref{thm:spikedModel},
% the continuity of $B(\eta, c)$, and the assumed continuity of
% $\eta$ and $L_2$. 
\end{proof}

\section{Examples of Decomposable Loss Functions}
\label{sec:exampl-decomp-loss}

% The rest of the section considers examples of decomposable loss
% functions. 
Many of the loss functions that appear in the literature 
are {\em Pivot-Losses}. 
They can
be obtained via the following common recipe:

\begin{defn} \label{def-pivot} {\bf Pivots.}  A {\em matrix pivot} is
  a matrix-valued function $\Delta(A, B)$ of 
  two real positive definitee matrices $A, B$ such that: (i)
  $\Delta(A,B)=0$ if and only 
  if $A=B$, (ii) $\Delta$ is orthogonally equivariant and (iii) 
  $\Delta$ respects block structure in the sense that
  \begin{align}
    \label{eq:p-inv}
    \Delta(OAO', OBO') & = O \Delta(A,B) O', \\
    \Delta( \oplus A^i, \oplus B^i) & = \oplus \Delta(A^i,B^i)
    \label{eq:p-block}
  \end{align}
for any orthogonal matrix $O$ of the appropriate dimension.
\end{defn}

Matrix pivots can be symmetric-matrix valued, for example $\Delta(A,B) = A - B$,
but need not be, for example $\Delta(A,B) = A^{-1} B - I$. 

\begin{defn} {\bf Pivot-Losses.}
  Let $g$ be a non-negative function
  of a symmetric matrix variable that is definite: $g(A)=0$ if and
  only if $A=0$, and orthogonally invariant:
$g(O \Delta O') = g(\Delta)$ for any orthogonal matrix $O$. 
%   Let $g$ be a real function
%   of a symmetric matrix variable such that: (i) $g\geq 0$, (ii) $g(A)=0$ if and
%   only if $A=0$, and (iii) $g$ is orthogonally invariant in the sense
%   that
% $g(O \Delta O') = g(\Delta)$ for any orthogonal matrix $O$. 
A symmetric-matrix valued pivot $\Delta$ induces an
orthgonally-invariant \textit{pivot loss}
\begin{equation}
  \label{eq:generic-loss} L(A,B) = g(\Delta(A,B)).
\end{equation} 
More generally, for any matrix pivot $\Delta$, set 
$|\Delta| = (\Delta'\Delta)^{1/2}$ and define 
%a pivot loss 
\begin{equation}
  \label{eq:asymmetric-loss}
  L(A,B) = g(|\Delta|(A,B)).
%\qquad \qquad   |\Delta| = (\Delta'\Delta)^{1/2}\,,
\end{equation}
% The
% orthogonally invariant loss
% \begin{equation}
%   \label{eq:generic-loss} L(A,B) = g(\Delta(A,B))\,,
% \end{equation} 
% in case $\Delta$ is a symmetric-matrix valued pivot, or 
% \begin{equation}
%   \label{eq:asymmetric-loss}
%   L(A,B) = g(|\Delta|(A,B)), \qquad \qquad 
%   |\Delta| = (\Delta'\Delta)^{1/2}\,,
% \end{equation}
% in case $\Delta$ is any pivot,
% is called a pivot-loss (w.r.t $\Delta$ and $g$).
\end{defn}

An orthogonally invariant function $g$ depends on its matrix argument
$\Delta$ or $|\Delta|$ only 
through its eigenvalues or singular values
$\delta_1,\ldots,\delta_p$. 
We abuse notation to write $g(\Delta)=g(\delta_1,\ldots,\delta_p)$.
Observe that if $g$ has either of the forms 
\begin{displaymath} \label{decom:eq}
  g(\delta_1,\ldots,\delta_p) =
  \sum_j g_1(\delta_j) \qquad \text{or} \qquad g(\delta_1,\ldots,\delta_p) = \max_j g_1(\delta_j),
\end{displaymath}
for some univariate $g_1$,
then the pivot loss $L(A,B)=g(\Delta(A,B))$ (symmetric pivot) or
$L(A,B)=g(|\Delta|(A,B))$ (general
pivot) is respectively sum- or max-decomposable. 
In case $\Delta$ is symmetric, the two definitions agree so long as
$g_1$ is an even function of $\delta$.
%
%For general pivots, we consider loss functions 
%\begin{equation}
%  \label{eq:asymmetric-loss}
%  L(A,B) = h(|\Delta|(A,B)), \qquad \qquad 
%  |\Delta| = (\Delta'\Delta)^{1/2}.
%\end{equation}
%Here, $h$ is orthogonally invariant and a function of the 
%\textit{singular values} $\sigma = (\sigma_j)$ 
%of $\Delta$. 
%If $h(\sigma) = \sum h_1(\sigma_j)$ or $\max h_1(\sigma_j)$,
%we obtain sum- or max-decomposable loss functions.
%Of course, if $\Delta$ is symmetric, then $\sigma_j = |\delta_j|$. 

\subsection{Examples of Sum-Decomposable Losses}

%Suppose $A$ and $B$ are jointly block diagonal as in 
%\eqref{eq:decomps}, and let $\Delta^i = \Delta(A^i, B^i)$. Write
%$\delta_1^i,\ldots,\delta_{n_i}^i$ for
%the eigenvalues or singular values of $\Delta^i$.
%If $L$ is a pivot-loss and $g$ is additive, then
%\begin{displaymath}
%  L(A,B) = \sum_i \sum_j g_1(\delta_j^i)\,
%\end{displaymath}
%so that $L$ is sum-decomposable.

There are different strategies to derive sum-decomposable pivot-losses.
First, we can use statistical discrepancies between the Normal distributions 
$\Nc(0,A)$
and $\Nc(0,B)$:

\begin{enumerate}
\item {\it Stein Loss} \cite{stein1956,dey1985estimation,konno1991estimation}:  
Stein's Loss 
is defined as 
\[
  L^{st} (A,B) = \tr(A^{-1}B - I) - \log( \det(B)/\det(A)).
\]
This is just twice the Kullback distance $D_{KL}( \Nc(0,B)|| \Nc(0,A))$.
Stein's loss is a pivot-loss with respect to
$\Delta(A,B) = A^{-1/2} B A^{-1/2}$ and 
$g(\Delta)=\tr(\Delta - I) - \log \det(\Delta)=\sum_i
g_1(\delta_i)$, 
where $g_1(\delta) = \delta - 1 - \log \delta.$

\item {\it Entropy/Divergence Losses:}
Because the Kullback discrepancy is not symmetric in its arguments,
we may consider two other losses: 
reversing the arguments we get {\it Entropy} loss
 $L^{ent}(A,B)  = L^{st}(B,A)$
\cite{SinhaGhosh,CJS:CJS91} and  summing the Stein and Entropy losses gives
{\it divergence}  loss:  \[L^{div}(A,B)  = L^{st}(A,B) + L^{st}(B,A)
  = \tr(A^{-1}B-I) + \tr(B^{-1}A - I),\]
see \cite{kubokawa1990estimating,gupta1995improved}.
Each can be shown to be sum-decomposable, following the same
argument as above.

\item {\it Bhattarcharya/Matusita Affinity} \cite{kailath1967divergence,matsusita1967}:
Let 
\[
  L^{\it aff}(A,B) = \tfrac{1}{2} \log \frac{|A+B|/2}{|A|^{1/2} |B|^{1/2}}\,.
\]
This measures the statistical distinguishability
of $\Nc(0,A)$ and $\Nc(0,B)$ based on independent observations,
since $L^{aff} = \frac{1}{2} \log(\int \sqrt{\phi_A}\sqrt{\phi_B})$
with $\phi_A$ and $\phi_B$ 
the densities of $\Nc(0,A)$ and $\Nc(0,B)$. Hence convergence of affinity
loss to zero 
is equivalent to convergence of the underlying densities in Hellinger or Variation distance.
This is a pivot-loss w.r.t
$\Delta(A,B)  = A^{-1/2} B A^{-1/2} $ and
\[
g(\Delta) = \tfrac{1}{4} \log(\det(2I + \Delta + \Delta^{-1})/4) = \sum_i
g_1(\delta_i),
\]
as is seen by setting 
$C = A^{-1/2}(A+B)B^{-1/2} $ and noting that $C'C = (2I + \Delta +
\Delta^{-1})$. Here, $g_1(\delta) = \tfrac{1}{4} \log(2 + \delta +
\delta^{-1})/4$. 

\item {\it Fr\'{e}chet Discrepancy} \cite{olkin1982distance,dowson1982frechet}:
Let $L^{fre}(A,B) = \tr(A+B - 2 A^{1/2}B^{1/2})$.
This measures the minimum possible mean-squared difference between zero-mean random vectors
with covariances $A$ and $B$ respectively. 
This is a pivot-loss w.r.t $\Delta(A,B) = A^{1/2}-B^{1/2}$, and
$g(\Delta) = \tr (\Delta^2)=\sum_i g_1(\delta_i)$ with 
$g_1(\delta) = \delta^2$.
\end{enumerate}

Second, we may obtain sum-decomposable pivot-losses $L(A,B)=g(\Delta(A,B))$ by
simply 
taking $g$ to be one of the standard matrix norms: 

\begin{enumerate}
\item {\it Squared Error Loss} \cite{james1961estimation,chen2010shrinkage,ledoitPeche,ledoit2012nonlinear}:
Let $L^{F,1}(A,B) = \| A - B \|_F^2$. This is a pivot-loss w.r.t $\Delta(A,B) = A - B$
and $g(\Delta) = \tr \Delta'\Delta=\sum_i g_1(\delta_i)$ with $g_1(\delta) = \delta^2$. 

\item {\it Squared Error Loss on Precision} \cite{haff1979estimation}:
Let $L^{F,2}(A,B) = \| A^{-1} - B^{-1} \|_F^2$.
This is a pivot-loss w.r.t $\Delta(A,B) = A^{-1} - B^{-1}$ and 
 $g(\Delta) = \tr \Delta'\Delta$.

\item {\it Nuclear Norm Loss.} 
Let $L^{N,1}(A,B) = \| A -B\|_*$ where $\| \Delta \|_*$ denotes the
nuclear norm of the matrix $\Delta$, i.e. the sum of its singular values.
This is a pivot-loss w.r.t $\Delta(A,B) = A - B$ and
$g(\Delta)=\sum_i |\delta_i|$.
\item Let $L^{F,3}(A,B)=\| A^{-1}B-I\|_F^2$. This is a pivot-loss w.r.t
   $\Delta(A,B)=A^{-1} B - I$. It
was studied in \cite{selliah1964estimation,haff1980empirical,Sharma1985} and later work.
\item Let $L^{F,7}(A,B)=\|\log(A^{-1/2}BA^{-1/2})\|_F^2$,
where $\log()$ denotes
the matrix logarithm\footnote{The matrix logarithm transfers the matrices from the Riemannian manifold
of symmetric positive semidefinite matrices to its tangent space at $A$.  
It can be shown that $L^{F,7}$ is the squared geodesic distance in this
manifold.
This metric between covariances has attracted attention, for example, in
diffusion tensor MRI \cite{lenglet2006statistics,dryden2009}.}
\cite{forstner1999metric,lenglet2006statistics}.
This is a pivot-loss w.r.t  \[\Delta(A,B)=\log(A^{-1/2}BA^{-1/2})\,.\]
\end{enumerate}

\subsection{Examples of Max-Decomposable Losses}

Max-decomposable losses arise by applying the operator norm (the maximal
singular value or eigenvalue of a matrix) to a
suitable pivot. 
%Let $A$ and $B$ be jointly block
%diagonalizable, with respective blocks $A^i, B^i$. Write
%$\sigma_j^i$ or $\delta_j^i$ for the singular values or eigenvalues 
%  $\Delta^i =
%\Delta(A^i,B^i)$. Since the operator norm of $\Delta
%\begin{displaymath}
%  L(A,B) = \max_{i,j} \sigma_j^i = \max_{i,j} |\delta_j^i|,
%\end{displaymath}
%the latter holding for symmetric pivots. 
Here are a few examples:

\begin{enumerate}
\item {\it Operator Norm Loss}  \cite{karoui2008operator}:
  Let $L^{O,1}(A,B) = \| A - B \|_{op}$. This is a pivot-loss 
w.r.t $\Delta(A,B) = A - B$ and $g(\Delta)=\|\Delta\|_{op} = 
\max_i \delta_i$.

\item {\it Operator Norm Loss on Precision:} 
Let $L^{O,2}(A,B) = \| A^{-1} - B^{-1} \|_{op}$. This is a pivot-loss w.r.t.
$\Delta(A,B)= A^{-1} - B^{-1}$.

\item {\it Condition Number Loss}:
Let $L^{O,7}(A,B) = \|  \log(A^{-1/2}BA^{-1/2})  \|_{op}$. This is a pivot-loss
w.r.t $\Delta(A,B) =   \log(A^{-1/2}BA^{-1/2})$, related to  \cite{won2012condition}.
In the spiked model discussed below, $L^{O,7}$ effectively measures
the condition number of $A^{-1/2}BA^{-1/2}$. 
\end{enumerate}

We adopt the systematic naming scheme
$L^{\mbox{norm},\mbox{pivot}}$ where $\mbox{norm} \in \{F,O,N\}$, and
$\mbox{pivot} \in \{ 1, \dots , 7 \}$. This set of 21 combinations
covers the previous matrix norm examples and adds some more. 
 Together with Stein's loss and the
others based on statistical discrepancy mentioned above,  we
arrive at a set of 26 loss functions, Table
\ref{Table:lossnames}, to be studied in this paper.

% Many other pivot-loss functions that are constructed in this manner are sum- or
% max-decomposable.  The
% previous examples all involved the $\ell_1$, $\ell_2$ or $\ell_\infty$ norm,
% applied to the eigenvalues or singular values of a pivot $\Delta(A,B)$ 
% satisfying $\Delta(A,A) = 0$.  We now adopt the systematic naming scheme
% $L^{\mbox{norm},\mbox{pivot}}$ where $\mbox{norm} \in \{F,O,N\}$, and
% $\mbox{pivot} \in \{ 1, \dots , 7 \}$.  The resulting 21 different possible
% combinations are all studied in this article. Together with Stein's loss and the
% other losses based on statistical discrepancy, mentioned above, in this paper we
% study 26 loss functions. 
% %Under this naming scheme, the
% %three examples immediately above are called $L^{F,1}$, $L^{F,2}$ and $L^{N,1}$,
% %respectively.
% The systematic notation we adopt for them is summarized in Table
% \ref{Table:lossnames}.

\begin{table}[h]
\begin{center}
\begin{tabular}{|c|c|c|c|}
\hline
         & \multicolumn{3}{|c|}{MatrixNorm} \\
\hline
Pivot  & Frobenius & Operator &  Nuclear  \\
\hline
$A-B$  & $L^{F,1}$ & $L^{O,1}$& $L^{N,1}$\\		    
   $A^{-1} - B^{-1} $ & $L^{F,2}$ & $L^{O,2}$ & $L^{N,2}$ \\
   $ A^{-1}B -I $  	& $L^{F,3}$ & $L^{O,3}$ & $L^{N,3}$ \\
   $B^{-1}A -I$    	& $L^{F,4}$ & $L^{O,4}$ & $L^{N,4}$ \\
   $ A^{-1}B + B^{-1}A-2I $ & $L^{F,5}$ & $L^{O,5}$ & $L^{N,5}$ \\
   $ A^{-1/2}BA^{-1/2} - I$  & $L^{F,6}$ & $L^{O,6}$ & $L^{N,6}$ \\
   $\log(A^{-1/2}BA^{-1/2}) $& $L^{F,7}$ & $L^{O,7}$ & $L^{N,7}$ \\
\hline
 & \multicolumn{3}{|c|}{Statistical Measures} \\

         & St &  Ent & Div \\
          \hline 
Stein & $L^{st}$ & $L^{ent}$ & $L^{\it div}$ \\
\hline
Affinity & \multicolumn{3}{|c|}{$L^{\it aff}$} \\
Fr\'echet & \multicolumn{3}{|c|}{$L^{\it fre}$} \\
\hline
\end{tabular}
\caption{Systematic notation for the 26 loss functions considered in this paper.}
\label{Table:lossnames}
\end{center}
\end{table}

% \section{The Spiked Covariance Model} \label{sec:spiked} 

% \begin{thm} ({\bf Spiked Covariance Model}) \label{thm:spikedModel}
%  % \TODO{add degeneracy}
% Let the theoretical (population) covariance have leading eigenvalues $\ell_i$, $i=1,\dots,r$
% obeying $\ell_i  >  \ell_+(\gamma)$.
% %and define the functions
% %\begin{eqnarray}
% %  \lambda(\ell) &=&   \ell \cdot (1
% %+ \gamma/(\ell-1)) \,, \nonumber \\
% %  c(\ell) &=&  \sqrt{ \frac{1 -
% %    \gamma/(\ell-1)^2 }{1 + \gamma/(\ell-1)} }\,.
% %    \label{c_ell:eq}
% %  \end{eqnarray}
% Then, as $n \goto \infty$, the leading 
% empirical eigenvalues obey
% \begin{eqnarray} \label{eig_displacement:eq}
%   \lambda_{i,n} \goto_P \,\lambda(\ell_i) ,  \qquad i=1,\dots, r\,,
% \end{eqnarray}
% where $\lambda(\ell)$ is defined in \eqref{lambda_of_ell:eq};
% See \cite{baik2006eigenvalues,bai2008central,Benaych-Georges2011}.
% Moreover, let $v_{i,n}$ denote the empirical eigenvector corresponding to $\lambda_{i,n}$ and
% let $u_{i,n}$ denote the theoretical  eigenvector corresponding to
% $\ell_i$. Suppose that the principal theoretical eigenvalues are
% distinct, namely $\ell_i \neq \ell_j$, $1 \leq i,j \leq r$. Then, as $n \goto \infty$, we have 
% \[
%     | \langle u_{i,n} , v_{j,n} \rangle | \goto_P \,\delta_{i,j} \cdot c(\ell_i)
%     \qquad 1 \leq i,j \leq r \,,
% \]
% where $c(\ell)$ is defined in \eqref{c_of_ell:eq};
% See \cite{paul2007asymptotics,Benaych-Georges2011}.
% \end{thm}

\section{Optimal Shrinkage for Decomposable Losses} \label{sec:OptShrink}

\subsection{Formally Optimal Shrinker}

Formula \eqref{eq:asy-formula} for the asymptotic loss has only been
shown to hold in the single spike model and only for a certain class
of nonlinearities $\eta$.  In fact, the same is true in the $r$-spike
model and for a much broader class of nonlinearities $\eta$.  To
preserve the narrative flow of the paper, we defer the proof, which is
more technical, to Section \ref{multspike:sec}. Instead, we proceed
under the single spike model, and simply assume
% assumption \Asy\, by simply assuming that \eqref{eq:asy-formula}
% holds, namely 
that $L_\infty(\ell|\eta)$ from \eqref{eq:asy-formula}
is the correct limiting loss, and draw conclusions on the optimal
shape of the shrinker $\eta$.

% Formula \eqref{eq:asy-formula} for the asymptotic loss has only been shown to
% hold under \Asy\, and \OneSpike, for certain nonlinearities $\eta$.
% In fact, the same is true under \Asy\, and \Spike,
% and for a much broader class of nonlinearities $\eta$. 
% To preserve the
% narrative flow of the paper, we defer the proof, which is largely technical, to 
% Section \ref{multspike:sec}. Instead, we proceed under assumption \Asy\, by simply assuming that  \eqref{eq:asy-formula}
% holds, namely  that $L_\infty(\ell|\eta)$ from \eqref{eq:asy-formula} is the
% correct limiting loss, and draw conclusions on the optimal shape of the
% shrinker $\eta$.

%as a function of the
%population eigenvalue $\ell$ has been shown to hold at hand, we proceed to 
% identify the
\begin{defn} \label{def:Optimal} {\bf Optimal Shrinker.}
  Let $L=\left\{ L_p \right\}_{p=1}^\infty$ be  a given loss family and let
 $L_\infty(\ell|\eta)$ be the asymptotic loss corresponding to a nonlinearity
 $\eta$, as defined in \eqref{eq:asy-formula}, under assumption \Asy\,.
 If $\eta^*$ satisfies
\begin{equation} \label{eq:OptEta}
    L_\infty(\ell| \eta^* ) =  \min_{\eta} L_\infty(\ell| \eta ), \qquad \forall
    \ell \geq 1\,,
\end{equation}
and for any $\eta\neq \eta^*$ there exists $\ell\geq 1$ with 
$L_\infty(\ell,\eta^*) < L_\infty(\ell,\eta)$, then
we say that $\eta^*$ is the {\em formally optimal}
shrinker for the loss family $L$ and shape factor $\gamma$, and denote
the corresponding 
shrinkage rule by $\lambda\mapsto \eta^*(\lambda\,;\,\gamma,L)$.
 \end{defn}

Below, we call formally optimal shrinkers simply ``optimal''.
By definition, the optimal shrinkage rule $\eta^*(\lambda\,;\,\gamma,L)$ 
is the unique admissible rule, in the asymptotic sense, among rules 
of the form $\hat{\Sigma}_\eta(S_{n,p}) = V \eta(\Lambda) V'$
in the single-spike model. 
%Simply put, the optimal rule is {\em
%asymptotically unique admissible}: 
In the single spiked model (and as we show
later, generally in the spiked model) one never regrets using the optimal
shrinker over any other (reasonably regular) univariate shrinker. 
%In view of the Lemma \ref{lem:asympLoss}, solving the problem (\ref{eq:OptEta})
%is simpler than one might expect, as it
%only involves optimization over 2-by-2 matrices:
%
In light of our results so far, an obvious characterization of an optimal shrinker is as follows.

\begin{thm}  \label{thm:Optimal} {\bf Characterization of Optimal Shrinker.}
Let $L=\left\{ L_p \right\}_{p=1}^\infty$ be  a loss family.
Define
\begin{eqnarray} \label{opt_obj1:eq}
F(\ell,\eta) = L_2 \big ( \begin{bmatrix}\ell & 0 \\ 0  &1 \end{bmatrix},
                                        \begin{bmatrix} 1+(\eta-1)c^2  & (\eta-1)cs  \\ (\eta-1)cs & 1+(\eta-1) s^2 \end{bmatrix}
                            \big )  \,.
\end{eqnarray}
Here, $c=c(\ell)$ and $s=s(\ell)$ satisfy $c^2(\ell) = \frac{1 - \gamma/(\ell-1)^2}{1+\gamma/(\ell-1)}$ and
$s^2(\ell) = 1-c^2(\ell)$.
Suppose that for any $\ell > \ell_+(\gamma)$, there exists a unique
  %and correspondingly $\lambda = \lambda(\ell) >  \lambda_+(\gamma)$.  
 minimizer 
\begin{eqnarray} \label{opt:eq}
  \eta^*(\ell) :=\text{argmin}_{\eta \geq 1} F(\ell,\eta) \,.
\end{eqnarray}
Further suppose that for every $1\leq \ell \leq \ell_+(\gamma)$ we have
$ \text{argmin}_{\eta \geq 1} G(\eta) = 1$, where
\begin{eqnarray} \label{opt_obj2:eq}
G(\ell,\eta) = L_2 \big ( \begin{bmatrix} \ell & 0 \\ 0  & 1 \end{bmatrix},
                                       \begin{bmatrix} 1 & 0 \\ 0  & \eta \end{bmatrix} 
                            \big ) \,.
\end{eqnarray}
Then the shrinker 
\[
  \eta^*(\lambda) = \begin{cases} 
    \eta^*(\ell(\lambda)) & \,\, \ell> \lambda_+(\gamma) \\
    1 & 1\leq \ell \leq \lambda_+(\gamma)
  \end{cases}\,,
\]
where $\ell(\lambda)$ is given by
\eqref{ell_of_lambda:eq}, is the optimal shrinker of the loss family $L$.

%Alternatively suppose $1 < \ell < \ell_+(\gamma)$. Consider the minimization
%$ \min_{\eta \geq 1} G(\eta)$, where
%\begin{eqnarray} \label{opt_obj2:eq}
%G(\eta,\ell) = L_2 \big ( \begin{bmatrix} \ell & 0 \\ 0  &1 \end{bmatrix},
%                                       \begin{bmatrix} 1 & 0 \\ 0  & \eta \end{bmatrix} 
%                            \big )  ;
%\end{eqnarray}
%if, for every  $\ell  \in [1, \ell_+)$, this minimum is achieved at $\eta=1$, then we say that  
%the 
% optimal shrinker collapses the  bulk to 1: i.e. the optimal shrinker satisfies 
% $\eta^*(\lambda) =  1$ for $\lambda < \lambda_+(\gamma)$.
\end{thm}

Many of the 26 loss families discussed in Section
\ref{sec:DecomposableLoss} admit a closed form expression for the optimal shrinker; see Table \ref{Table:formulas}. For
others, we computed the optimal
shrinker numerically, by implementing in software a solver for the simple 
scalar optimization problem
 \eqref{opt:eq}.
Figure \ref{fig:OptimalShrinkers}
portrays the  optimal shrinkers for our 26 loss functions. We refer readers
interested in computing
specific individual shrinkers to our reproducibility advisory at the
bottom of this paper, and invite the reader to explore the code supplement \cite{SDR},
consisting of online resources and code we offer.

\begin{figure}[h!]
\begin{center}
  \includegraphics[width=6.5in,clip]{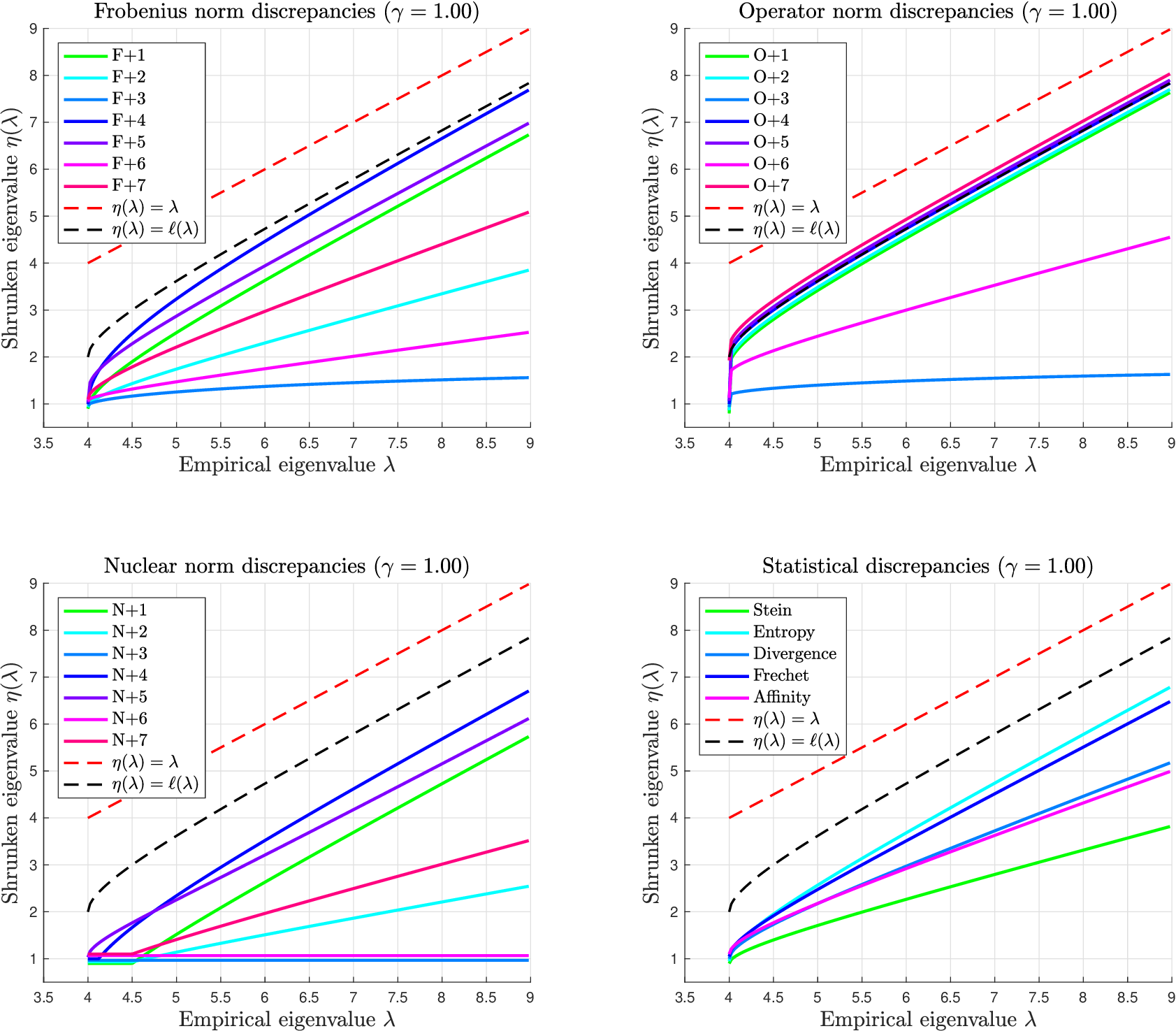}
\caption{Optimal Shrinkers for 26 Component Loss Functions for  $\gamma=1$ and
  $4 \leq \lambda \leq 10$.
Upper Left: Frobenius-norm-based losses; Lower Left: Nuclear-Norm based losses;
Upper Right: Operator-norm-based losses;  Lower Right: Statistical Discrepancies.
(Color online; curves jittered in vertical axis to avoid overlap.)
The supplemental article 
\cite{SI} contains an larger version of these plots.
 {\it Reproducibility advisory:} 
The code supplement \cite{SDR}
includes a script that reproduces any one of these individual curves.
 }
\label{fig:OptimalShrinkers}
\end{center}
\end{figure}

\subsection{Optimal Shrinkers Collapse the Bulk}

We first observe that, for any of the 26 losses considered, the optimal shrinker
collapses the bulk to $1$. The following lemma is proved in the supplemental article 
\cite{SI}:

\begin{lemma} \label{bulk:lem}
  Let $L$ be any of the 26 losses mentioned in Table \ref{Table:lossnames}. Then
  the rule $\eta^{**}(\ell)= 1$ is unique asymptotically
  admissible on $[1,\ell_+(\gamma)]$, namely,  for 
  every $\ell\in[1,\ell_+(\gamma)]$ we have
$\E L(\ell,\eta)\geq
  L(\ell,\eta^{**})$,
  with strict inequality for at least one point in $[1,\ell_+(\gamma)]$.
\end{lemma}

As part of the proof of Lemma %\ref{bulk:lem}
4, in Table %\ref{Table:g1g2} 
6 in the 
supplemental article  \cite{SI}, we explicitly calculate the fundamental loss
function  $G(\ell,\eta)$ of \eqref{opt_obj2:eq} for many of the loss families
discussed in this paper.
\\~\\
To determine the optimal shrinker $\eta^*(\lambda\,;\,\gamma,L)$ for each of our
loss functions $L$, it
therefore remains to determine the map $\lambda\mapsto \eta^*(\lambda)$ or
equivalently $\ell\mapsto \eta^*(\lambda(\ell))$ 
 only for $\ell>\ell_+(\gamma)$. This is our next
task.

\subsection{Optimal Shrinkers by Computer} 

The scalar optimization problem \eqref{opt:eq} is easy to solve numerically, so that
one can always compute the optimal shrinker at any desired value $\lambda$.
In the code supplement \cite{SDR} we provide Matlab code to compute the optimal nonlinearity for each of the
26 loss families discussed. In the sibling problem of singular value shrinkage
for matrix denoising,   \cite{GavishDonohoSingShrink} demonstrates numerical
evaluation of optimal shrinkers for the Schatten-$p$ norm, where analytical
derivation of optimal shrinkers appears to be impossible.

\subsection{Optimal Shrinkers in Closed Form}

We were able to obtain simple analytic formulas for the optimal shrinker
$\eta^*$ in each of 18 
 loss families from  Section \ref{sec:DecomposableLoss}. While 
the optimal shrinkers are of course functions of the empirical eigenvalue
$\lambda$, in the interest of space, we state the lemmas and provide the formulas in terms of 
 the quantities $\ell$, $c$ and $s$. To calculate any
of the nonlinearities below for a specific empirical eigenvalue $\lambda$, use the
following procedure:
\begin{enumerate}
  \item If $\lambda\leq \lambda_+(\gamma)$ set $\eta^*(\lambda)=1$. Otherwise:
  \item Calculate  $\ell(\lambda)$ using \eqref{ell_of_lambda:eq}.
  \item Calculate $c(\lambda)=c(\ell(\lambda))$ using
    (\ref{c_of_ell:eq}) and
% \eqref{c_ell:eq} and
    \eqref{ell_of_lambda:eq}.
  \item Calculate $s(\lambda)=s(\ell(\lambda))$ using $s(\ell) = \sqrt{1-c^2(\ell)}$.
  \item Substitute $\ell(\lambda)$, $c(\lambda)$ and $s(\lambda)$ into the
    formula provided to get $\eta^*(\lambda)$.
    %\footnote{We stress that these formulas describe the optimal
    %shrinkers only for $\lambda>\lambda_+(\gamma)$; all optimal shrinkers
    %discussed  
  %shrink to $1$ any $\lambda\leq\lambda_+(\gamma)$.}.
\end{enumerate}

\def\Ldiv{\sqrt{\dfrac{\ell^2c^2+\ell s^2}{c^2+\ell s^2}}}
\def\Lfre{\left( \sqrt{\ell}c^2+s^2 \right)^2}
\def\Laff{\dfrac{(1+c^2)\ell+s^2}{1+c^2+\ell s^2}}
\def\Fone{\ell c^2 + s^2}
\def\Ftwo{\dfrac{\ell}{c^2+\ell s^2}}
\def\Fthree{\dfrac{\ell c^2 + \ell^2 s^2}{c^2 +\ell^2s^2}}
\def\Ffour{\dfrac{\ell^2c^2+s^2}{\ell c^2 + s^2}}
\def\Fsix{1+\dfrac{(\ell-1)c^2}{\left( c^2+\ell s^2 \right)^2}}
\def\Osix{1+\dfrac{\ell-1}{c^2+\ell s^2}}
\def\NOne{\max\left(1+(\ell-1)(1-2s^2)\,,\,1\right)}
\def\Ntwo{\max\left(\dfrac{\ell}{c^2+(2\ell-1)s^2}\,,\,1\right)}
\def\Nthree{\max\left( \dfrac{\ell}{c^2 + \ell^2s^2}\,,\,1 \right)}
\def\Nfour{\max\left( \dfrac{\ell^2 c^2 + s^2}{\ell}\,,\,1 \right)}
\def\Nsix{\max\left( \dfrac{\ell-(\ell-1)^2c^2s^2}{(c^2+\ell s^2)^2}\,,\,1 \right)}

\begin{table}[h!]
\begin{center}
{\renewcommand{\arraystretch}{2.4}
\begin{tabular}{|c|c|c|c|}
\hline
Pivot         & \multicolumn{3}{|c|}{MatrixNorm} \\

  & Frobenius & Operator &  Nuclear  \\
\hline
   $A-B$  & $\Fone$ & $ \ell $ & $\NOne$\\ \hline		    
   $A^{-1} - B^{-1} $ & $\Ftwo$ & $\ell$ & $\Ntwo$ \\ \hline
   $ A^{-1}B -I $  	& $\Fthree$ & N/A  & $\Nthree$ \\ \hline
   $B^{-1}A -I$    	& $\Ffour$ & N/A & $\Nfour$ \\ \hline
   $ A^{-1/2}BA^{-1/2} - I$  & $\Fsix$ & $\Osix$ & $\Nsix$ \\ \hline
 & \multicolumn{3}{|c|}{Statistical Measures} \\ 
         & St &  Ent & Div \\
          \hline 
Stein & $\Ftwo$ & $\Fone$ & $\Ldiv$ \\
\hline
Fr\'echet & \multicolumn{3}{|c|}{$\Lfre$} \\
\hline
Affine & \multicolumn{3}{|c|}{$\Laff$} \\
\hline
\end{tabular}}
\caption{
  Optimal shrinkers $\eta^*(\lambda)$ for 18 of the loss families $L$ discussed.
  Values shown are shrinkers for $\lambda>\lambda_+(\gamma)$. All shrinkers obey
  $\eta^*(\lambda)=1$ for $\lambda\leq \lambda_+(\gamma)$.  Here, $\ell$, $c$
  and $s$ depend on $\lambda$ (and implicitly on $\gamma$) according to
  \eqref{ell_of_lambda:eq}, \eqref{c_of_ell:eq} and $s=\sqrt{1-c^2}$.  In cases
  marked ``N/A'' the optimal shrinker does not seem to admit 
  a simple closed form,
  but can be easily calculated numerically.
}
\label{Table:formulas}
\end{center}
\end{table}

The closed forms we provide are summarized in Table
\ref{Table:formulas}. 
Note that
$\ell$, $c$ and
$s$ refer to 
the functions $\ell(\lambda)$, $c(\ell(\lambda))$ and
$s(\ell(\lambda))$.
These formulae are formally derived in a sequence of lemmas that are stated and
proved 
%These formulas are derived in the following  sequence of lemmas,  proved in 
in the supplemental article \cite{SI}. 
The proofs also show that these optimal shrinkers are unique, as in each case
the optimal shrinker is shown to be the unique minimizer, as in \eqref{opt:eq}, 
of \eqref{opt_obj1:eq}.
We make some remarks on these optimal shrinkers by focusing first on
operator norm loss for covariance and precision matrices:
% \begin{lemma} ({\bf Operator Norms.}) \label{lem:operatorNormNonlinearity}
% For the direct operator norm loss $L^{O,1}$ and
% the operator norm loss on precision matrices  $L^{O,2}$, we have
\begin{equation}
\label{eq:operatorOpt}
\eta^*(\lambda ; \gamma , L^{O,1}) =  \eta^*(\lambda ; \gamma , L^{O,2}) = \left\{ 
            \begin{array}{ll}
                \ell,  & \ell >  \ell_+(\gamma) \\
                1   ,  & \ell \leq  \ell_+(\gamma) 
                \end{array} \right . .
\end{equation}
%\end{lemma}
%
This asymptotic relationship reflects the classical fact that in
finite samples,  
the top empirical eigenvalue is
always biased upwards of the underlying population eigenvalue
\cite{vanderVaart1961,cacoullos1965}. 
Formally defining the (asymptotic) bias as 
\[bias(\eta, \ell) = \eta(\lambda(\ell)) - \ell\,,\]
we have $bias(\lambda(\ell),\ell) > 0$.
The 
formula $\eta^*(\lambda) = \ell$ shows that the optimal nonlinearity for
operator norm loss  is what we
might simply call a {\it debiasing} transformation, mapping each empirical
eigenvalue back to the value of its ``original'' population eigenvalue, and the
corresponding shrinkage estimator $\hat{\Sigma}_\eta$ uses each {\em sample}
eigenvectors with its corresponding {\em population} eigenvalue.
In words, within the top branch of (\ref{eq:operatorOpt}),
the effect of {\it operator-norm optimal shrinkage is to debias the top eigenvalue}:
\[
    bias(\eta^*(\cdot ; \gamma , L^{O,1}),\ell) = bias(\eta^*(\cdot ; \gamma
    , L^{O,2}),\ell) = 0, \qquad \forall \ell > \ell_+(\gamma).
\]
On the other hand, within the bottom branch, 
the effect is to {\it shrink the bulk to 1}.
In terms of  Definition \ref{def:bulk-shrinkers} 
%\ref{def:collapseBulk},
we see that  $\eta^*$ is a bulk shrinker, 
%collapse the vicinity 
but not a neighborhood bulk shrinker.

One might expect  asymptotic debiasing from {\it every} loss function,
but, perhaps surprisingly, precise asymptotic debiasing is exceptional. 
In fact,  none of the other optimal nonlinearities in Table
\ref{Table:formulas} is precisely debiasing.

In the supplemental article \cite{SI} we also provide a detailed investigation
of the large-$\lambda$ asymptotics of the optimal shrinkers, including their
asymptotic slopes, asymptotic shifts and asymptotic percent improvement.

\section{Beyond Formal Optimality} \label{multspike:sec}

The shrinkers we have derived and analyzed above are formally optimal, as in 
Definition 
\ref{def:Optimal}, in the sense that they minimize the formal expression 
$L_\infty(\ell|\eta)$.
So far we have only shown that formally optimal shrinkers actually 
minimize the asymptotic loss (namely, are asymptotically unique admissible)
in the single-spike
case, under assumptions \Asy\, and \OneSpike, and only over neighborhood
bulk 
shrinkers.

In this section, we show that formally optimal shrinkers
in fact minimize the asymptotic loss in the general Spiked Covariance Model,
namely under assumptions \Asy\, and \Spike, and over a large class of
bulk shrinkers, which are possibly not neighborhood bulk shrinkers.

We start by establishing the rank $r$ analog of Lemma
\ref{jointBlock-perlim}. 
For a vector $\ell \in \R^r$, let $\Delta_r(\ell) =
\text{diag}(\ell_1, \ldots, \ell_r)$. 

\begin{lemma} \label{lem:rank-r-decomp}
  Assume that $\Sigma$ and $\hat{\Sigma}$ are fixed matrices with 
\begin{align*}
  spec(\Sigma) 
    & = [(\ell_1, \ldots, \ell_r, 1, \ldots, 1), (u_1,\ldots,u_p)] \\
  spec(\hat \Sigma)
    & = [(\eta_1, \ldots, \eta_r, 1, \ldots, 1), (v_1,\ldots,v_p)]. 
\end{align*}
Let $U_r$ and $V_{r}$ denote the $p$-by-$r$ matrices consisting of the top
  $r$ eigenvectors of $\Sigma$ and $\hat{\Sigma}$  respectively.
Suppose that $[U_r \ V_r]$ has full rank $2r$, and consider the $QR$
decomposition
\begin{equation*}
  [U_r \ V_r] = Q R,
\end{equation*}
where $Q$ has $2r$ orthonormal columns and the $2r \times 2r$ matrix
$R$ is upper triangular.
Let $R_2$ denote the $2r \times r$ submatrix formed by the last $r$
columns of $R$. 
Fill out $Q$ to an orthogonal matrix $W = [ Q \ Q^\perp]$.
Then in the transformed basis we have the simultaneous block
decompositions
\begin{alignat}{2}
    W' \Sigma W 
     & = \Sigma_{2r}^\circ \oplus I_{p-2r}, \qquad &
     \Sigma_{2r}^\circ & = \Delta_r(\ell) \oplus I_r  \label{eq:A2r} \\    
  W' \hat \Sigma W 
     & = \hat \Sigma_{2r}^\circ \oplus I_{p-2r}, & 
     \hat \Sigma_{2r}^\circ & = I_{2r} + R_2 \Delta_r(\eta-1) R_2'. \label{eq:B2r}
\end{alignat}
\end{lemma}
\begin{proof}

We start with observations about the structure of $Q$ and $R$.
Since the first $r$ columns of $Q$ are
identically those of 
  $U_r$, we let $Z_r$ be the $n$-by-$r$ matrix such that 
   $Q=[U_r\,Z_r]$.
 For the same reason, $R$ 
% Also note that the upper-diagonal matrix $R$ in \eqref{W1_QR:eq} 
has the block structure 
  \[
    R= \begin{bmatrix} I_{r\times r} & R_{12} \\ 0_{r\times r} & R_{22}\,
    \end{bmatrix}\,,
  \]
  where the matrices $R_{12}$ and $R_{22}$ satisfy 
  $
    V_{r} = U_r R_{12} + Z_r R_{22}\,,
  $
  so that
  \begin{eqnarray}
    R_{12} =  U_r'\, V_{r} \qquad
    R_{22} = Z_{r}'\, V_{r} \label{two:eq}\,.
  \end{eqnarray}
    Since $V_{r}$ has orthogonal columns, we have
\begin{align}
  I_r  = V_{r}'\, V_{r} & = R_{12}'\, R_{12} + R_{22}' \, R_{22}
  \notag \\
  R_{22}'\, R_{22} &=  I - R_{12}'\, R_{12}\,.
  \label{eq:R22}
\end{align}

%Fill out $Q$ to an orthogonal matrix $W = [Q \ Q^\perp]$.
Let $H$ be a $p \times r$ matrix whose columns lie in the column span
of $Q$ and let $\Delta$ be an $r \times r$ diagonal matrix.
Observe that
\begin{align*}
  W'(I + H \Delta H') W 
    & = I + W' H \Delta H' W \\
    & = (I_{2r} + Q' H \Delta H' Q) \oplus I_{p-2r} 
     = C_{2r} \oplus I_{p-2r},
\end{align*}
say, since the columns of $Q^\perp$ are orthogonal to those of $H$.

By analogy to \eqref{eq:rankone}, we may write
\begin{equation} \label{eq:rank-r}
   \Sigma = I + U_r (\Delta_r(\ell) - I_r) U_r', \qquad \quad 
   \hat \Sigma = I + V_r (\Delta_r(\eta) - I_r) V_r'
\end{equation}
and so both of the form $I + H\Delta H'$, with $H = U_r$ and $V_r$
respectively. We find that
\begin{equation*}
  Q'U_r =
  \begin{bmatrix}
    I_r \\ 0
  \end{bmatrix}, \qquad 
  Q'V_r =
  \begin{bmatrix}
    R_{12} \\ R_{22}
  \end{bmatrix} = R_2,
\end{equation*}
We can then compute the value of $C_{2r}$ in the two cases to be given
by $\Sigma_{2r}^\circ$ and $\hat \Sigma_{2r}^\circ$ respectively,
which establishes  
\eqref{eq:A2r} and \eqref{eq:B2r}, and hence the lemma.
\end{proof}

% We can apply these considerations to the representations
% \eqref{eq:rank-r} of $\Sigma$ and $\hat \Sigma$ to compute the form of
% $C_{2r}$. 

% Let $\Sigma_{2r}, \hat \Sigma_{2r}$ denote the
% representations of $\Sigma, \hat \Sigma$ in the basis provided by the
% columns of $W_{2r}$; equivalently
% \begin{equation*}
%   \Sigma_{2r}
%    = W_{2r}' \Sigma W_{2r}, \qquad 
%   \hat \Sigma_{2r}
%    = W_{2r}' \hat \Sigma W_{2r}.
% \end{equation*}

% We arrive at
% \begin{align}
%   \Sigma_{2r} & = \Delta_r(\ell) \oplus I_r, \notag \\
% %  \Sigma_{2r}^{(2)} & = \oplus_{k=1}^r A(\ell_k),  \\
%   \hat \Sigma_{2r} & = I_{2r} + R^{(1)} (\Delta_r(\eta) - I_r)
%   R^{(1) \prime}, \label{eq:R-form} \\ 
% %  \hat \Sigma_{2r}^{(2)} & = I_{2r} + R^{(2)} (\Delta_r(\eta) - I_r)
% %  R^{(2) \prime}, \\
%   R^{(1)} & =
%   \begin{bmatrix}
%     R_{12} \\ R_{22}
%   \end{bmatrix},  \notag
%   % \qquad \quad 
%   % R^{(2)} = \Pi_{2r}' R^{(1)}.
% \end{align}

% \bigskip
% \bigskip
% Let $\hat \Sigma_\eta = \sum_{i=1}^p \eta(\lambda_i) v_i v_i'$ and
% consider the `rank-aware' modification
% \begin{equation*}
%   \hat \Sigma_{\eta,r} 
%    = \sum_{i=1}^r \eta(\lambda_i) v_i v_i' + \sum_{i=r+1}^p  v_i v_i'.
% \end{equation*}

We intend to apply 
Lemma \ref{lem:rank-r-decomp} to $\Sigma$ and $\hat \Sigma = \hat
\Sigma_{\eta,r}$, the ``rank-aware'' modification \eqref{eq:rank-aware}
of the estimator $\hat \Sigma_\eta$ in \eqref{eq:scalar-shrink}.
Assume now that $\hat \Sigma$ and the 
$p \times r$ matrix $V_{r,n}$ formed by the top eigenvectors of $V$ are random.

\begin{lemma} \label{lem:rank_2r}
  The rank of $[U_r \ V_{r,n}]$ equals $2r$ almost surely.
\end{lemma}
\begin{proof}
  Let $\Pi_r(V)$ be the projection that picks out the first $r$
  columns of an orthogonal matrix $V$. For a fixed $r$-frame $U_r$, we
  consider the event
  \begin{equation*}
    A = \{ V \in O_p ~:~ \text{rank}([U_r \ \Pi_r(V)]) < 2r \}\,,
  \end{equation*}
  where $O_p$ is the group of orthogonal $p$-by-$p$ matrices.
Let $P_\Sigma(d \Lambda, d V)$ denote the joint distribution of
eigenvalues $\Lambda = \text{diag}(\lambda_1, \ldots, \lambda_p)$ and
eigenvectors $V$ when $S \sim W_p(n,\Sigma)$.
As shown by \cite{james1954}, $P_\Sigma$ is absolutely continuous with
respect to $\nu_p \times \mu_p$, the product of Lebesgue measure on
$\R^p$ and Haar measure on $O(p)$. Since $\mu_p(A) = 0$,
%\textbf{[This needs a couple of sentences of justification?]}
 it follows that $P_\Sigma(A) = 0$.
\end{proof}

%\footnotesize \textit{Remark.} We can avoid this by writing
%$[U \ V] = Q R$ with $Q$ and $n \times (r+r')$ matrix with $r + r'
%\leq 2r$ orthogonal columns and then doing the same argument. Since it
%will still be the case that $R_{22} R_{22}' \stackrel{a.s.}{\to}$ a
%rank $r$ matrix, we get that $r' \stackrel{a.s.}{\to} r$. 
%But perhaps it is clearer to have all the action happen on a set of
%probability one from the start.
%\normalsize

\begin{lemma}
  \label{lem:rank-aware-lim}
  Adopt models \Asy\, and \Spike\, with
$\ell_1, \ldots, \ell_r > \ell_+(\gamma)$. Suppose the scalar
nonlinearity $\eta$ is continuous on $(\lambda_+(\gamma),\infty)$.
For each $p$ there exists w.p. 1 an orthogonal change of basis $W$ such that
\begin{equation}
    \label{eq:O-sum}
  W' \Sigma W = \Sigma_{2r} \oplus I_{p-2r}, \qquad 
  W' \hat \Sigma_{\eta, r} W = \hat \Sigma_{2r} \oplus I_{p-2r},
\end{equation}
where the $2r \times 2r$ matrices $\Sigma_{2r}, \hat \Sigma_{2r}$
satisfy
\begin{equation}
   \label{eq:AB-form}
  \Sigma_{2r} = \oplus_{i=1}^r A(\ell_i), \qquad
  \hat \Sigma_{2r} \ \stackrel{a.s.}{\to} \  
      \oplus_{i=1}^p B(\ell_i,\eta),
\end{equation}
and the $2 \times 2$ matrices $A(\ell), B(\ell,\eta)$ are defined at
(\ref{eq:ABforms}).

Suppose also that the family $L=\left\{ L_p \right\}$ of loss functions is
orthogonally invariant and sum- or max- decomposable, and that
$B \to L_{2r}(A, B)$ is continuous. Then
\begin{equation} \label{eq:rankaware-lim}
  L_p(\Sigma, \hat \Sigma_{\eta,r})
   \ \stackrel{a.s.}{\to} \
   \left(\sum/\max\right)_{i=1, \ldots r} \ L_2(A(\ell_i), B(\ell_i, \eta))).
\end{equation}
If $\eta$ is a neighborhood bulk shrinker,
%collapses the vicinity of the bulk to $1$, 
then $L_p(\Sigma, \hat \Sigma_\eta)$ also has this limit a.s.
\end{lemma}
This is the rank $r$ analog of Lemma \ref{lem:asymp-loss}.
The optimal nonlinearity $\eta^*$ is continuous on $[0,\infty)$ for all losses except
the operator norm ones, for which $\eta^*$ is continuous except at
$\lambda = \lambda_+(\gamma)$. 
Our result (\ref{eq:AB-form}) requires only continuity on
$(\lambda_+(\gamma),\infty)$ and so is valid for all 26 loss
functions, 
% Some of these will be discussed separately at
% REF below.
% Thus the deterministic limit \eqref{eq:rankaware-lim} holds for the
% rank-aware $\hat \Sigma_{\eta,r}$ for all these non-operator loss
% functions. 
as is the deterministic limit \eqref{eq:rankaware-lim}  for the
rank-aware $\hat \Sigma_{\eta,r}$.
However, as we saw earlier,  only the nuclear norm based loss
functions yield optimal  functions that 
are neighborhood bulk shrinkers.
%only a small number of optimal shrinkage functions 
%collapse the vicinity of the bulk to 1.  
To show that
\eqref{eq:rankaware-lim} holds for $L_p(\Sigma, \hat \Sigma_\eta)$ for
most other important shrinkage functions, some further
work is needed -- see Section \ref{sec:removing-rank-aware}
below.

\begin{proof}
  We apply Lemma \ref{lem:rank-r-decomp} to $\Sigma$ and $\hat
  \Sigma_{\eta,r}$
on the set of probability 1 provided by Lemma \ref{lem:rank_2r}.
First, we rewrite (\ref{eq:B2r}) to show the subblocks of $R$:
\begin{equation*}
  \hat \Sigma_{2r}^\circ 
   = I_{2r} +
   \begin{bmatrix}
     R_{12} \\ R_{22}
   \end{bmatrix}
   \Delta_r(\eta^{(n)} -1)
   \begin{bmatrix}
     R_{12}' & R_{22}'
   \end{bmatrix},
\end{equation*}
where we write $\eta^{(n)} = (\eta(\lambda_{1,n}), \ldots,
\eta(\lambda_{r,n}))$ to show explicitly the dependence on $n$.
The limiting behavior of $R$ may be derived from (\ref{two:eq}) and
(\ref{eq:R22}) along with spiked model properties
 (\ref{eig_displacement:eq}) and (\ref{eq:uv-incon}),
%Theorem \ref{thm:spikedModel}, 
so we have\footnote{For simplicity, we chose the $QR$ decomposition to
  make the sign  of $s(\ell_i)$ positive.},
as $n\to\infty$, 
\begin{align}
  R_{12} =  U_r' \,V_{r,n} \  & \ \goto_{a.s}  \Delta_r(c) \notag \\
  R_{22} \,R_{22}' =  I - R_{12} \,R_{12}' & \ \goto_{a.s.}
  \Delta_r(s^2)  \label{eq:Rcge}\\
  R_{22} & \ \goto_{a.s.}  \Delta_r(s) .  \notag
\end{align}
Here $c = (c(\ell_1), \ldots, c(\ell_r))$ and $s = (s(\ell_1), \ldots,
s(\ell_r))$. 

Again by (\ref{eig_displacement:eq})
%by Theorem \ref{thm:spikedModel}, 
$\lambda_{i,n} \to_{a.s.} \lambda(\ell_i) >
\lambda_+(\gamma)$ and so continuity of $\eta$ above
$\lambda_+(\gamma)$ assures that
$\Delta_r(\eta^{(n)}-1) \to \Delta_r(\eta -1)$, where 
$\eta = (\eta_i)$ and $\eta_i = \eta(\lambda(\ell_i))$. 
Together with (\ref{eq:uv-incon}),
%Theorem \ref{thm:spikedModel} 
%Putting this into \eqref{eq:B2r}, 
we obtain simplified structure in the limit,
\begin{equation} 
  \label{eq:block-limit}
  \hat \Sigma_{2r}^\circ \to_{a.s.} I_{2r} + 
  \begin{bmatrix}
    \Delta_r( (\eta-1)c^2 ) & \Delta_r( (\eta-1)cs ) \\
    \Delta_r( (\eta-1)cs )  & \Delta_r( (\eta-1)s^2 )
   \end{bmatrix}.
\end{equation}
To rewrite the limit in block diagonal form, 
let $\Pi_{2r}$ be the permutation matrix corresponding to the
permutation defined by 
  \[
    (1,\ldots,2r)\mapsto \left( 1,r+1,2,r+2,3,\ldots,2r \right)\, .
  \]
Permuting rows and columns in (\ref{eq:A2r}) and
(\ref{eq:block-limit}) using $\Pi_{2r}$ to obtain 
\begin{align*}
  \Sigma_{2r} & := \Pi_{2r}' \Sigma_{2r}^\circ \Pi_{2r} = \oplus_{i=1}^r
  A(\ell_i),    \\
  \hat \Sigma_{2r} 
      & := \Pi_{2r}' \hat \Sigma_{2r}^\circ \Pi_{2r} 
     \to_{a.s.} \oplus_{i=1}^p B(\ell_i,\eta),
\end{align*}
we obtain \eqref{eq:AB-form}. 
Using \eqref{eq:O-sum}, the orthogonal invariance and sum/max
decomposability, along with the continuity of $L_{2r}(A,\cdot)$, we
have
\begin{align*}
  L_p(\Sigma_p, \hat \Sigma_{\eta,r}) 
   & = L_p( \Sigma_{2r} \oplus I_{p-2r}, \hat \Sigma_{2r} \oplus
   I_{p-2r}) \\
   & = L_{2r}( \Sigma_{2r}, \hat \Sigma_{2r}) \\
   & = L_{2r}( \Pi_{2r}' \Sigma_{2r} \Pi_{2r}, \Pi_{2r}' \hat
   \Sigma_{2r} \Pi_{2r}) \\
   & \to_{a.s.} L_{2r}( \oplus_{i=1}^r A(\ell_i), \oplus_{i=1}^p
   B(\ell_i,\eta)) \\
   & = \left(\sum/\max\right)_{i=1, \ldots r} \ L_2(A(\ell_i),
   B(\ell_i, \eta))), 
\end{align*}
which completes the proof of Lemma \ref{lem:rank-aware-lim}.
\end{proof}

\subsection{Removing the rank-aware condition}
\label{sec:removing-rank-aware}

In this section we prove Proposition \ref{prop:removing-rank-aware} below, 
whereby the asymtotic losses coincide for a given 
estimator sequence $\hat \Sigma_\eta$ and the rank-aware versions
$\hat \Sigma_{\eta,r}$. This result is plausible because of two
observations:
\begin{enumerate}
  \item Null eigenvalues {\em stick to the bulk}, i.e. for $i \geq r+1$, most 
eigenvalues $\lambda_{in} \leq \lambda_+(\gamma)$ and the few exceptions
are not much larger. Hence, if $\eta$ is a continuous bulk shrinker,
we expect $\hat 
\Sigma_\eta$ to be close to $\hat \Sigma_{\eta,r}$, 

% since most, or
% all, $\eta(\lambda_{in})=1$ for $i \geq r+1$. 

\item under a suitable continuity assumption on the loss functions
$L_p$, $L(\Sigma,\hat\Sigma_\eta)$ should then be close to 
$L(\Sigma, \hat \Sigma_{\eta,r})$.
\end{enumerate}
\noindent Observation 1 is fleshed out in two steps.
The first step is eigenvalue comparison:
The sample eigenvalue
$\lambda_{in}$ arise as eigenvalues of $X X'/n$ when $X$ is a
$p_n$-by-$n$ matrix whose rows are i.i.d draws from
$\mathcal{N}(0,\Sigma_{p_n})$.
Let $\Pi:\mathbb{R}^{p_n}\to\mathbb{R}^{p_n-r}$ denote the projection on 
the last $p_n-r$ coordinates in $\mathbb{R}^{p_n}$ and
let $\mu_{1n} \geq \dots \geq \mu_{p_n-r,n}$ denote 
the eigenvalues of 
$\Pi X (\Pi X)'/n$. By the Cauchy interlacing Theorem 
(e.g. \cite[p. 59]{bhat97}), we have 
\begin{equation}
  \label{eq:interlace} 
\lambda_{in}\leq \mu_{i-r,n} \qquad \text{for} \  r+1\leq i\leq p_n,
\end{equation}
where the $(\mu_{in})$ are the eigenvalues of a white Wishart matrix 
$W_{p_n -r}(n,I)$. 

The second step is a bound on eigenvalues of a white Wishart that exit
the bulk. 
Before stating it, we return to an important detail
introduced in the Remark concluding Section \ref{assumptions:subsec}.

Definition \ref{def:bulk-shrinkers} of a bulk shrinker depends on the
parameter $\gamma = \lim p/n$ through $\lambda_+(\gamma)$. 
Making that dependence explicit, we obtain a bivariate function 
$\eta(\lambda,c)$. 
In model \Asy and in the $n$-th problem, we might use $\eta(\lambda,
c_n)$ either with $c_n = \gamma$ or $c_n = p/n$. For Proposition
\ref{prop:evals-sticky} below, it will be more natural to use the
latter choice. We also modify Definition \ref{def:bulk-shrinkers} as
follows. 

\begin{defn}
  \label{def:joint-bulk-shrinker}
We call $\eta: [0,\infty) \times (0,1] \to [1,\infty)$ a
\textit{jointly continuous bulk shrinker} if $\eta(\lambda,c)$ is 
 jointly continuous in $\lambda$ and $c$, satisfies
$\eta(\lambda, c) =  1 $ for $\lambda \leq \lambda_+(c)$ and is 
dominated: $\eta(\lambda, c) \leq  M \lambda $ for some $M$ and all
$\lambda$.
% \begin{alignat*}{2}
%   \eta(\lambda, c) \
%     & = \ 1 \qquad \quad && \text{for } \lambda \leq \lambda(c) \\
%   \eta(\lambda, c) \ 
%     &  \leq \ M \lambda && \text{for some } M \ \text{and all } \lambda.
% \end{alignat*}
\end{defn}

The following result is proved in \cite[Theorem 2(a)]{IMJ2017}.

\begin{proposition}
  \label{prop:evals-sticky}
  Let $(\mu_{in})_{i=1}^{N}$ denote the sample
  eigenvalues of a matrix distributed as $W_{N}(n,I)$, with $N/n
  \to \gamma > 0$. Suppose that 
%$\eta: [1,\infty) \to   [1,\infty)$ is a
$\eta(\lambda,c)$ is a jointly continuous bulk shrinker
and that $c_n - N/n = O(n^{-2/3})$. Then for $q > 0$, 
  \begin{equation}
    \label{eq:evals-sticky}
    \| \eta(\mu_{in},c_n) - 1 \|_{\ell_q(\R^{N})} \to_P 0.
  \end{equation}
\end{proposition}

% \begin{proposition}
%   \label{prop:evals-sticky}
%   Let $(\mu_{in})_{i=1}^{p_n}$ denote the sample
%   eigenvalues of a matrix distributed as $W_{p_n}(n,I)$, with $p_n/n
%   \to \gamma > 0$. Suppose that 
% %$\eta: [1,\infty) \to   [1,\infty)$ is a
% $\eta(\lambda,c)$ is a jointly continuous bulk shrinker
% and that $c_n - p_n/n = O(n^{-2/3})$. Then for $q > 0$, 
%   \begin{equation}
%     \label{eq:evals-sticky}
%     \| \eta(\mu_{in},c_n) - 1 \|_{\ell_q(\R^{p_n})} \to_P 0.
%   \end{equation}
% \end{proposition}

The continuity assumption on the loss functions may be formulated as
follows. 
Suppose that $A, B_1, B_2$ are $p$-by-$p$ 
positive definite matrices, with $A$ satisfying assumption 
\Spike\  and 
$\text{spec}(B_k) = [(\eta_{ki}),(v_i)]$, thus $B_1$ and $B_2$ have
the same eigenvectors. Set $\eta_1 = \max \{ \eta_{11}, \eta_{21}
\}$. 
We assume that for some $q \in [1,\infty]$ and
some continuous function $C(\ell_1, \eta_1)$ not depending on $p$, we
have
\begin{equation}
  \label{eq:C-cond}
  |L_p(A,B_1) - L_p(A,B_2) | 
   \leq C(\ell_1, \eta_{1}) \,
     \| \eta_1 - \eta_2\|_{\ell_q(\R^p)} 
\end{equation}
whenever $\| \eta_1 - \eta_2\|_{\ell_q(\R^p)} \leq 1$.
Condition (\ref{eq:C-cond}) is satisfied for all 26 of the loss
functions of Section \ref{sec:DecomposableLoss}, as is verified in
Proposition 1 in SI.

In the next proposition we adopt the convention that estimators 
$\hat \Sigma_\eta$ of (\ref{eq:scalar-shrink}) and 
$\hat \Sigma_{\eta,r}$ of (\ref{eq:rank-aware}) are constructed with a
jointly continuous bulk shrinker, which we denote $\eta(\lambda, c_n)$.

\begin{proposition}
  \label{prop:removing-rank-aware}
Adopt models \Asy\, and \Spike.
Suppose also that the family $L=\{ L_p \}$ of loss functions is
orthogonally invariant and sum- or max- decomposable, and satisfies
continuity condition (\ref{eq:C-cond}).
If $\eta(\lambda, c_n)$ is a jointly continuous bulk shrinker with
$c_n = p_n/n$, then
\begin{equation*}
  L_p(\Sigma, \hat\Sigma_{\eta})
   - L_p(\Sigma, \hat\Sigma_{\eta,r}) \to_P 0,
\end{equation*}
and so $L_p(\Sigma, \hat\Sigma_{\eta})$ converges in probability to
the deterministic asymptotic loss (\ref{eq:rankaware-lim}).
\end{proposition}

\begin{proof}
In the left side of (\ref{eq:C-cond}), substitute
$A = \Sigma, B_1 = \hat \Sigma_\eta$ and $B_2 = \hat \Sigma_{\eta,r}$.
By definition, $\hat \Sigma_{\eta}$ and $\hat \Sigma_{\eta,r}$ share the same
eigenvectors. 
The components of $\eta_1 - \eta_2$ then satisfy
\begin{equation*}
  \eta_{1i} - \eta_{2i} =
  \begin{cases}
    \eta(\lambda_{in}, c_n) - 1 \qquad & i \geq r+1 \\
    0 & 1 \leq i \leq r.
  \end{cases}
\end{equation*}
We now use (\ref{eq:interlace}) to
compare the eigenvalues $\lambda_{in}$ of the spiked model to
those of a suitable white Wishart matrix to which Proposition
\ref{prop:evals-sticky} applies.
% The $\lambda_{in}$ arise as eigenvalues of $X X'/n$ when $X$ is a
% $p_n$-by-$n$ matrix whose rows are i.i.d draws from
% $\mathcal{N}(0,\Sigma_{p_n})$.
% Let $\Pi:\mathbb{R}^{p_n}\to\mathbb{R}^{p_n-r}$ denote the projection on 
% the last $p_n-r$ coordinates in $\mathbb{R}^{p_n}$ and
% let $\mu_{1n} \geq \dots \geq \mu_{p_n-r,n}$ denote 
% the eigenvalues of 
% $\Pi X (\Pi X)'/n$. By the Cauchy interlacing Theorem, we have 
%$\lambda_{in}\leq \mu_{i-r,n}$ for $r+1\leq i\leq p_n$. 
The function $\eta^{\uparrow}(\mu, c) = \max \{ \eta(\lambda, c), 1 \leq
\lambda \leq \mu \}$ and is non-decreasing and
jointly continuous. Hence
$\eta(\lambda_{in},c_n)
  \leq \eta^\uparrow(\lambda_{in},c_n)
  \leq \eta^\uparrow(\mu_{i-r,n},c_n)$, 
% \begin{equation*}
%   \eta(\lambda_{in})
%   \leq \eta^\uparrow(\lambda_{in})
%   \leq \eta^\uparrow(\mu_{i-r,n}).
% \end{equation*}
and so
\begin{equation*}
  \sum_{i=r+1}^p [\eta(\lambda_{in},c_n) - 1]^q 
  \leq \sum_{j=1}^{p-r} [\eta^\uparrow(\mu_{jn},c_n) - 1]^q,
\end{equation*}
with a corresponding bound for $q = \infty$. From continuity condition
(\ref{eq:C-cond}),
\begin{equation*}
  | L_p(\Sigma, \hat\Sigma_{\eta})
   - L_p(\Sigma, \hat\Sigma_{\eta,r})|
  \leq C(\ell_1, \eta(\lambda_{1n},c_n)) \ 
    \| \eta^\uparrow(\mu_{jn},c_n)-1\|_{\ell_q(\R^{p-r})}.
\end{equation*}
The constant $C(\ell_1, \eta(\lambda_{1n},c_n))$ remains bounded by
(\ref{eig_displacement:eq}).
% Apply Proposition \ref{prop:evals-sticky} with
% $N = p_n - r$ to the eigenvalues of $W_{p_n-r}(n,I)$, noting that 
% $c_n - (p_n -r)/n = r/n = O(n^{-2/3})$. Hence the $\ell_q$ norm
% converges to $0$ in probability.
 The $\ell_q$ norm converges to $0$
in probability, applying Proposition \ref{prop:evals-sticky} 
to the eigenvalues of $W_{p_n-r}(n,I)$,
with $N = p_n - r$, noting that $c_n - N/n = r/n = O(n^{-2/3})$.
\end{proof}

% \bigskip
% \bigskip

%\newpage
\subsection{Asymptotic loss for discontinuous optimal shrinkers}
\label{sec:asy-loss-disc}

%Lemma \ref{lem:operatorNormNonlinearity},
Formula (\ref{eq:operatorOpt}) showed that the optimal shrinker
$\eta^*(\lambda, \gamma)$ for operator norm losses
$L^{O,1},L^{O,2}$ is \textit{discontinuous} at $\ell = \ell_+(\gamma) = 
1 + \sqrt \gamma$. 
In this section, we show that when $\eta^*$ is used, a deterministic
asymptotic loss exists for $L^{O,1}$, but \textit{not} for $L^{O,2}$.
The reason will be seen to lie in the behavior of the optimal
component loss 
$F_*(\ell) = L_2[A(\ell),B(\ell,\eta^*)]$.
Indeed, calculation based on (\ref{opt_obj1:eq}), (\ref{eq:operatorOpt})
shows that for $\ell \geq \ell_+$, 
\begin{equation*}
  F_*(\ell)
   = \left[ \frac{\ell^a \gamma (\ell-1)}{\ell -1+\gamma}
  \right]^{1/2} 
   \to F_*(\ell_+) 
   =
   \begin{cases}
     \sqrt{\gamma} & a = 1 \\[4pt]
     \dfrac{\sqrt{\gamma}}{1+\sqrt{\gamma}} & a = -1
   \end{cases}
\end{equation*}
as $\ell \downarrow \ell_+$, where indices $a = 1$ and $-1$ correspond
to $F_*^{O,1}$ and $F_*^{O,2}$ respectively.
% \begin{alignat}{3}
%   F_*^{O,1} (\ell)
%   & = \left[ \frac{\gamma \ell (\ell-1)}{\ell -1+\gamma}
%   \right]^{1/2} 
%   &&  \geq \ \ \ \sqrt \gamma && = F_*^{O,1} (\ell_+),  \label{eq:cpt+1} \\
%   F_*^{O,2} (\ell) 
%   & = \left[ \frac{\gamma (\ell-1)}{\ell(\ell -1+\gamma)}
%   \right]^{1/2} 
%   &&  \leq \frac{\sqrt \gamma}{1 + \sqrt \gamma} && = F_*^{O,2}
%   (\ell_+),
%    \label{eq:cpt-1}
% \end{alignat}
Importantly, $F_*^{O,1}$ is strictly increasing on $[\ell_+,\infty)$
while $F_*^{O,2}$ is strictly decreasing there.

\begin{proposition}
  \label{prop:op-discty}
  Adopt models \Asy\, and \Spike\, with
$\ell_r > \ell_+(\gamma)$.
Consider the optimal shrinker $\eta^*(\lambda, \gamma_n)$ with $\gamma_n =
p_n/n$ given by (\ref{eq:operatorOpt})  for
both $L^{O,1}$ and $L^{O,2}$. For $L^{O,1}$, the asymptotic loss is well defined:
\begin{equation}
  \label{eq:O-1}
    \| \hat \Sigma_\eta - \Sigma \|_\infty 
  - \| \hat \Sigma_{\eta,r} - \Sigma \|_\infty \to_P 0.
\end{equation}
However, for $L^{O,2}$, 
\begin{equation}
  \label{eq:O-2}
    \| \hat \Sigma^{-1}_\eta - \Sigma^{-1} \|_\infty 
  - \| \hat \Sigma^{-1}_{\eta,r} - \Sigma^{-1} \|_\infty 
  \stackrel{\mathcal{D}}{\to} W.
\end{equation}
where $W$ has a two point distribution in which
\begin{equation*}
  W =
  \begin{cases}
     F_*^{O,2}(\ell_+) - F_*^{O,2}(\ell_r) 
%    \sqrt \gamma/(1+\sqrt \gamma) - F_*^{O,2}(\ell_r) 
         & \text{with prob } 1 - F_1(0) \\
    0    & \text{otherwise},
  \end{cases}
\end{equation*}
and $F_1(0) = \mathbb{P}\{TW_1 \leq 0\}$ for a real Tracy-Widom
variate $TW_1$ \cite{TW1996}.
\end{proposition}

Roughly speaking, there is positive limiting probability that the
largest noise eigenvalue will exit the bulk distribution, and in such
cases the corresponding component loss $F_*(\ell_+)$ -- which is due to
noise alone -- exceeds the
largest component loss due to any of the $r$ spikes, namely
$F_*(\ell_r)$. Essentially, this occurs because precision losses
$L^{\{O,F,N\},2}(a\Sigma,a \hat \Sigma)$ \textit{decrease}
as signal strength $a$ increases. The effect is not seen for
$L^{\{F,N\},2}$ because the optimal shrinkers in those cases are
continuous at $\ell_+$ !

\begin{proof}
For the proof, write $\| \cdot \|$ for $\| \cdot \|_\infty$.
%and $\delta_{in} = \eta(\lambda_i,\gamma_n) - 1$.
Let $W = [W_1 \ W_2 ]$ be the orthogonal change of basis matrix
constructed in Lemma \ref{lem:rank-aware-lim}, with $W_1$ containing the
first $2r$ columns.
We treat the two losses $L^{O,1}$ and $L^{O,2}$ at once using
an exponent $a = \pm 1$,
and write $\eta^a(\lambda)$ for $\eta^a(\lambda, \gamma_n)$. Thus, let 
\begin{align*}
\Delta & = \Delta_n =  \hat \Sigma^a_\eta - \hat \Sigma^a_{\eta,r} 
= \sum_{i=r+1}^p [\eta^a(\lambda_i)-1] v_i v_i',\\
\intertext{and observe that the loss of the rank-aware estimator}
%\begin{equation*}
  \Psi & = \Psi_n = \hat \Sigma^a_{\eta,r} - \hat \Sigma^a
     = \sum_{i=1}^r [ \eta^a(\lambda_i)-1] v_i v_i'
       - \sum_{i=1}^r (\ell^a_i -1) u_i u_i'
% \end{equation*}
\end{align*}
lies in the column span of $W_1$.
We have $\hat \Sigma^a_\eta - \Sigma^a = \Psi_n + \Delta_n$, and  
the main task will be to show that for $a = \pm 1$,
\begin{equation}
  \label{eq:op-sum}
  \| \Psi_n + \Delta_n \| = \max (\| \Psi_n\|, \| \Delta_n \|) + o_P(1).
\end{equation}

Assuming the truth of this for now, let us derive the proposition. The
quantities of interest in (\ref{eq:O-1}), (\ref{eq:O-2}) become
\begin{align*}
  \| \hat{\Sigma}_{\eta}^a - \hat{\Sigma}^a \| - 
  \| \hat{\Sigma}_{\eta,r}^a - \hat{\Sigma}^a \|
  & = \| \Psi_n + \Delta_n \| - \| \Psi_n \| \\
  & = \max (\| \Delta_n \| - \| \Psi_n \|, 0) + o_P(1).
\end{align*}
First, note from Lemma \ref{lem:rank-aware-lim} that
\begin{equation} \label{eq:psilim}
  \| \Psi_n \| \to_{a.s.} \max_{1\leq i \leq r} F_*(\ell_i).
  % \begin{cases}
  %   F_*^{O,1}(\ell_1) \qquad & a = 1 \\
  %   F_*^{O,2}(\ell_r) & a = -1.
  % \end{cases}
\end{equation}
Observe that for \textit{both} $a = 1$ and $-1$,
\begin{equation*}
  \| \Delta_n \| = \max_{i \geq r+1} |\eta^{* a}( \lambda_{in}) -1| 
= |\eta^a(\lambda_{r+1,n}) -1|.
%  \| \Delta_n \| 
%\to_P c_a I(TW > 0) 
\end{equation*}
The rescaled noise eigenvalue
$p^{2/3}(\lambda_{r+1,n} - \lambda_+(\gamma_n)) 
\stackrel{\mathcal{D}}{\to} \sigma(\gamma) W$ has a limiting real
Tracy-Widom distribution with scale factor $\sigma(\gamma) > 0$
\cite[Prop. 5.8]{bggm11}.
Hence, using the discontinuity of the optimal shrinker $\eta^*$, and
the square root singularity from above
\begin{equation*}
  \eta^*(\lambda_{r+1,n},\gamma_n) =
  \begin{cases}
    \ell_+(\gamma_n) + O_P(p^{-1/3}) 
%    1 + \sqrt \gamma + O_P(p^{-2/3}) 
    & \lambda_{r+1,n} > \lambda_+(\gamma_n) \\
    1 & \lambda_{r+1,n} \leq  \lambda_+(\gamma_n).
  \end{cases}
\end{equation*}
Consequently, recalling that $ F_*(\ell_+) = |(1+ \sqrt \gamma)^a -1|
$, we have
\begin{equation}
  \label{eq:dellim}
    \| \Delta_n \|  \to_{P}  F_*(\ell_+) I( TW > 0).
    % \| \Delta_n \|  \to_{P} c_a I( TW > 0), \quad \qquad
    %             c_a  = |(1+ \sqrt \gamma)^a -1| = F_*(\ell_+).
  % \begin{split}
  %   \| \Delta_n \| & \to_{P} c_a I( TW > 0) \\
  %               c_a & = |(1+ \sqrt \gamma)^a -1| = F_*(\ell_+).
  % \end{split}
\end{equation}

For $L^{O,1}$, with $a = 1$, $F_*(\ell)$ is strictly increasing and so
from (\ref{eq:psilim}) and (\ref{eq:dellim}), we obtain
$\| \Psi_n \| \geq \| \Delta_n \| + o_P(1)$ and hence (\ref{eq:O-1}).
For $L^{O,2}$, with $a = -1$, $F_*(\ell)$ is strictly decreasing and
so on the event $TW > 0$,
\begin{equation*}
  \| \Delta_n \| - \| \Psi_n \| 
   \stackrel{\mathcal{D}}{\to} F_*(\ell_+) - F_*(\ell_r) > 0,
\end{equation*}
which leads to (\ref{eq:O-2}) and hence the main result.

It remains to  prove (\ref{eq:op-sum}).
For a symmetric block matrix,
% $M = \begin{pmatrix}
%   A & B \\ B' & C
% \end{pmatrix}$,
%we have
% \begin{equation}
%   \label{eq:mbd}
%   \max( \| A \|, \| C \|)
%    \leq \| M \| 
%    \leq \max( \| A \|, \| C \|) +  \| B \|.
% \end{equation}
\begin{equation}
  \label{eq:mbd}
  \max( \| A \|, \| C \|)
   \leq \bigg \|
   \begin{pmatrix}
         A & B \\ B' & C
   \end{pmatrix}
        \bigg\| 
   \leq \max( \| A \|, \| C \|) +  \| B \|.
\end{equation}
Apply this to $ W'(\Psi + \Delta) W$ with
% where $W = [W_1 \  W_2]$ is an
% orthogonal $p \times p$ matrix with the columns of the $p \times 2$
% submatrix $W_1$ being a basis for $\mathcal{W}_2 = \text{span} \{u_1,
% v_1 \}$. This gives
\begin{align*}
  A_n & = W_1' (\Psi + \Delta) W_1, \\
  B_n & = W_1' (\Psi + \Delta) W_2 = W_1' \Delta W_2, \\
  C_n & = W_2' (\Psi + \Delta) W_2 = W_2' \Delta W_2,
\end{align*}
since $\Psi W_2 = 0$.
Hence
\begin{equation} 
\label{eq:Psi-Delt-bd}
  \| \Psi_n + \Delta_n \| 
   = \max( \| A_n \|, \| C_n \|) + O_P( \| B_n \|).
\end{equation}

We  now show that $\| \Delta W_1 \| \to_P 0$. 
Using notation from Lemma \ref{lem:rank-r-decomp},
\begin{equation*}
  W_1 = [U_r \ \  V_r] R^{-1} 
      = [U_r \ \ (V_r - U_r R_{12}) R_{22}^{-1}].
\end{equation*}
Since $\Delta v_k = 0$ for $k = 1, \ldots, r$,
\begin{equation*}
  \| \Delta W_1 \| \leq \| \Delta U_r \| (1 + \| R_{12} R_{22}^{-1} \| ).
\end{equation*}
From (\ref{eq:Rcge}), we have 
$\| R_{12} R_{22}^{-1} \| \to \| \Delta_r(c/s) \| =
c(\ell_1)/s(\ell_1)$, and hence is bounded.
Observe that
%Since $\Delta v_k = 0$ for $k = 1, \ldots, r$, 
% \begin{equation*}
%   \| \Delta W_1 \| = \| \Delta U_r \| 
%     \leq \sqrt{r} \max_{k = 1, \ldots, r} \| \Delta u_k \|_2,
% \end{equation*}
$\Delta u_k = \sum_{i=r+1}^p \delta^a_{in} (v_i' u_k) v_i$,
where  we have set $\delta_{in} = \eta(\lambda_i, \gamma_n) - 1$. 
Note from (\ref{eq:operatorOpt}) that $\delta_{in} = 0$ unless
$\lambda_i > \lambda_+(\gamma_n)$. 
With $N_n = \# \{ i \geq r+1  ~:~ \lambda_{in} > \lambda_+(\gamma_n)
\}$, we then have
\begin{equation} \label{eq:delUr}
  \| \Delta U_r \| 
   \leq \sqrt{r} \max_{k = 1, \ldots, r} \| \Delta u_k \|_2 
   \leq \sqrt{r} \| \Delta \| N_n \max_{k \leq r; i > r} |v_i' u_k|.
\end{equation}
% \begin{align*}
%   \| \Delta U_r \| 
%   & \leq \sqrt{r} \max_{k = 1, \ldots, r} \| \Delta u_k \|_2 \\
%   & \leq \sqrt{r} \| \Delta \| N_n \max_{k \leq r; i > r} |v_i' u_k|.
% \end{align*}

From (\ref{eq:dellim}) we have $\| \Delta_n \| = O_P(1)$. 
Since each $v_i, i > r$ is uniformly distributed on $S^{p-1}$,
a simple union bound based on (\ref{eq:unifbound}) below yields
\begin{equation}
  \label{eq:maxubd}
  \max_{i>r, k\leq r} (v_i'u_k)^2 
    = O_P \left( \frac{\log p}{p} \right).
\end{equation}

It remains to bound $N_n$. 
From the interlacing inequality (\ref{eq:interlace}),
% argument in the previous section, if 
% $(n \mu_{jn})$ are the eigenvalues of a white Wishart $W_{p_n
%   -r}(n,I)$ matrix, then
\begin{equation*}
  N_n  % = \# \{ i \geq r+1  ~:~ \lambda_{in} > \lambda_+(\gamma_n) \}
      \leq \tilde N_n
      = \# \{ j \geq 1  ~:~ \mu_{jn} > \lambda_+(\gamma_n) \},
\end{equation*}
where $\{ \mu_{jn} \}$ are the eigenvalues of a white Wishart matrix 
$W_{p_n -r}(n,I)$. This quantity is bounded in 
\cite[Theorem 2(c)]{IMJ2017}, which says that
%Let $N_n = \# \{ i \geq r+1  ~:~ \lambda_{in} > \lambda(\gamma_n) \}$.
%[\textbf{Interlacing!}]
% Since $|\delta_{in}^a| \leq |\delta_{r+1,n}^a|$ for $i \geq r+1$, we
% obtain 
% We know from \texttt{arXiv-note} that $\mathbb{E} \tilde N_n \to c \in
% (0,\infty)$.  
% Hence
% \begin{equation*}
%   \| \Delta u_k \|_2 \leq |\delta^a_{r+1,n}|\cdot \tilde N_n \max_{i \geq
%     r+1} | v_i'u_k|. 
% \end{equation*}
% Since each $v_i, i > r$ is uniformly distributed on $S^{p-1}$,
% a simple union bound based on (\ref{eq:unifbound}) below yields
% \begin{equation}
%   \label{eq:maxubd}
%   \max_{i>r, k\leq r} (v_i'u_k)^2 
%     = O_P \left( \frac{\log p}{p} \right).
% \end{equation}
% We apply Proposition 1(c) of \cite{IMJ2016} to conclude that 
$\tilde N_n = O_p(1)$.
In more detail, we make the correspondences $N \leftarrow p_n -r,
\gamma_N \leftarrow (p_n -r)/n$ and $c_N \leftarrow p_n/n$ so that 
$c_N - \gamma_N = r/n = o(n^{-2/3})$ and obtain
$E \tilde N_n \to c_0 \doteq 0.17$.

From (\ref{eq:delUr}) and the preceding two paragraphs,
we conclude that $\| \Delta U_r \| = O_P( p^{-1/2} \sqrt{\log p})$ and
so $\| \Delta W_1 \| \to_P 0$.

% Since also $|\delta^a_{r+1,n}| \leq \sqrt \gamma (1 + o_p(1))$ 
% %and $\tilde N_n = O_p(1)$ 
% we conclude that $\| \Delta W_1 \| \stackrel{p}{\to} 0$. 

Returning to (\ref{eq:Psi-Delt-bd}), we deduce now that $\| B_n \|
\leq \| \Delta W_1 \| \to_P 0$.
From the definition of $W_1$ we have 
$\| W_1' \Psi W_1 \| = \| \Psi \|$ and hence the inequalities
\begin{equation*}
  | \| A_n \| - \| \Psi_n \| | 
    \leq \| W_1' \Delta W_1 \|
    \to_P 0.
%, \qquad \quad   \| C \| \leq \| \Delta \|.
\end{equation*}
Now observe that $\| C_n \| \leq \| \Delta_n \|$.
Apply (\ref{eq:mbd}) to $ W' \Delta W$ to get
\begin{equation*}
  \| \Delta_n \| 
  \leq \| C_n \| + \| W_1' \Delta W_1 \| + \| W_2' \Delta W_1 \|,
\end{equation*}
and hence that $\| C_n \| \geq \| \Delta_n \| - o_P(1)$.
Thus $\| C_n \| = \| \Delta_n \| + o_P(1)$. 
Inserting these results into (\ref{eq:Psi-Delt-bd}), we obtain
\begin{equation*}
  \| \Psi_n + \Delta_n \| 
  = \max( \| A_n \|, \| C_n \|) + o_P(1) 
  = \max( \| \Psi_n \|, \| \Delta_n \|) + o_P(1),
\end{equation*}
which completes the proof of (\ref{eq:op-sum}), and hence of
Proposition \ref{prop:op-discty}.
\end{proof}

Finally, we record a concentration bound for the
uniform distribution on spheres. While more sophisticated results are
known \cite{ledoux}, an elementary bound suffices for us.

\begin{lemma}
  If $U$ is uniformly distributed on $S^{n-1}$ and $u \in S^{n-1}$ is
  fixed, then for $M > 0$ and $n \geq 4$,
\begin{equation}
    \label{eq:unifbound}
% P\Big( |\langle U, v \rangle | \geq 2 \sqrt{(M \log n) /n}
% \Big) 
P\big( |\langle U,u \rangle | \geq 2 \sqrt{M n^{-1} \log n}
\big) 
% P\Big( |\langle U, v \rangle | \geq 2 \sqrt{\frac{M \log n}{n}}
% \Big) 
\leq \sqrt{\pi/2} \cdot n^{1/2 - M}.
\end{equation}
\end{lemma}
\begin{proof}
  Since $U_1^2 := \langle U,u \rangle^2$ has the
  $\text{Beta}(\frac{1}{2}, \frac{n-1}{2})$ distribution,
\begin{equation*}
 P(U_1^2 \geq a)
  \leq B(\tfrac{1}{2}, \tfrac{n-1}{2})^{-1}
    \int_a^1 t^{-\tfrac{1}{2}} (1-t)^{\tfrac{n-3}{2}} dt
  \leq \gamma_n (1-a)^{\tfrac{n}{2}-1},
\end{equation*}
where by Gautschi's inequality \citep[(5.6.4)]{NIST:DLMF,Olver:2010:NHMF}
\begin{equation*}
%  \gamma_n = \frac{B(\hf,\hf)}{B(\hf, \tfrac{n-1}{2})}
  \gamma_n = B(\hf,\hf)/B(\hf, \tfrac{n-1}{2})
    = \sqrt{\pi} \Gamma(\tfrac{n}{2})/\Gamma(\tfrac{n-1}{2})
    < \sqrt{\pi n/2}
\end{equation*}
Since $(1-x/m)^m < e^{-x}$ for $x, m > 0$, and $4/n \geq 2/(n-2)$ for
$n \geq 4$, 
\begin{equation*}
  P(U_1^2 \geq 4 M n^{-1} \log n)
   < \sqrt{\pi n/2} \left( 1 - \frac{M \log n}{n/2 - 1} \right)^{n/2 -
     1} 
   < \sqrt{\pi/2} \cdot n^{1/2 - M}. \qedhere
\end{equation*}
\end{proof}

\section{Optimality Among Equivariant Procedures} \label{optimality:sec}

The notion of optimality in asymptotic loss, with which we have been concerned
so far,
is relatively weak. Also, the class of covariance estimators we have considered,
namely procedures that apply the {\em same} univariate shrinker to all empirical
eigenvalues, is fairly restricted.

% Naturally
%  the reader will at this point suppose that the procedures we proposed are optimal {\it only}
%  with respect to their respective limiting loss, and {\it only} within the
%  class of estimators of the form  $ \hat{\Sigma}(S) = V \eta(\Lambda) V'$,
% where $\eta$ is a single nonlinearity applied in turn to 
% each of the empirical eigenvalues $\lambda_i$, and where, as before, the columns
% of $V$ are the empirical eigenvectors.

Consider the much broader class of orthogonally-equivariant procedures for
covariance estimation 
 \cite{stein1986,LinPerlman1985,muirhead1987developments}, in which estimates
 take the form $\hat{\Sigma} = V\,\Delta \,V'$. Here, $\Delta =\Delta(\Lambda)$
 is {\em any} diagonal
 matrix that depends on the empirical eigenvalues $\Lambda$ in possibly a 
 more complex way than the simple 
 scalar element-wise shrinkage $\eta(\Lambda)$ we have considered so far. One might imagine that
  the extra freedom available with more general shrinkage rules
 would lead to improvements in loss, relative to our optimal scalar nonlinearity;
 certainly the proposals of  \cite{stein1986,LinPerlman1985,ledoit2012nonlinear}
 are of this more general type.

The smallest achievable loss by 
any orthogonally equivariant procedure is obtained with the ``oracle''
procedure $\hat{\Sigma}^{oracle} =  V \,\Delta^{oracle} \,V'$, where
% $\Delta^{oracle}$ is 
% chosen to be the minimizer
\begin{equation}
  \label{eq:oracle}
        \Delta^{oracle} = \argmin_{\Delta} L(\Sigma,V \,\Delta\, V'),
\end{equation}
the minimum being taken over diagonal matrices with diagonal entries $\geq 1$.
Clearly, this optimal performance is not attainable, since the minimization problem
explicitly demands perfect knowledge of $\Sigma$, precisely the object that we
aim to recover. This knowledge is never
available to us in practice -- hence the label {\it oracle}\footnote{The oracle procedure does not attain zero loss since it is ``doomed'' to use
  the eigenbasis of the empirical covariance, which is a random basis corrupted
by noise, to estimate the population covariance.}.
Nevertheless, this optimal performance is a legitimate benchmark.

Interestingly, at least for the popular Frobenius and Stein
losses,  our optimal nonlinearities 
$\eta^*$ deliver oracle-level performance --
% asymptotically.in the popular Frobenius case $L^{F,1}$ and in the Stein
% Loss case $L^{st}$,  the optimal nonlinearities 
% $\eta^*$, which we have derived, deliver oracle-level performance --
asymptotically.
To state the result, recall expression (\ref{opt_obj1:eq}) for these
losses:
$F(\ell,\Delta) = L_2(A(\ell),B(\ell,\Delta)).$
% \begin{equation}
%   \label{eq:Fdef}
%     F(\ell,\Delta) = L_2(A(\ell),B(\ell,\Delta)).
% \end{equation}
% \begin{align}
%   \label{eq:Fdef}
%   F(\ell,\Delta) 
%   & = L_2(A(\ell),B(\ell,\Delta)) \\
%   & =
%   \begin{cases}
%     (\ell -1)^2 - 2(\ell -1)\Delta c^2 + (\Delta-1)^2 & \text{for }
%     L^{F,1} \\
%     \hf (\ell^{-1}-1) + \hf (\Delta-1)(c^2/\ell + s^2) - \hf
%     \log(\Delta/\ell) & \text{for } L^{St}
%   \end{cases}
% \end{align}

\begin{thm}  \label{lem:oracleAsyLoss}
({\bf Asymptotic optimality among all equivariant procedures.})
Let $L$ denote either the direct Frobenius loss $L^{F,1}$ or the Stein loss
$L^{st}$. 
Consider a problem sequence satisfying assumptions 
\Asy\,and \Spike. 
% For $\Sigma_{p_n}$  and $\hat{\Sigma}_{p_n}^{oracle}$,  the $p_n$-by-$p_n$ 
% matrices in our sequence of statistical estimation problems,
We have
\begin{equation*} \label{oracleAsyLoss:eq}
  \lim_{n\to\infty}  L_{p_n}(\Sigma,\hat{\Sigma}^{oracle}) 
     =_P      L_\infty(\ell_1\,\ldots,\ell_r | \eta^*) 
     = \sum_{i=1}^r F(\ell_i, \eta^*),
    \end{equation*}
  % \begin{align*} \label{oracleAsyLoss:eq}
  % \lim_{n\to\infty}  L_{p_n}(\Sigma_{p_n},\hat{\Sigma}_{p_n}^{oracle}) 
  %   &  =_P      L_\infty(\ell_1\,\ldots,\ell_r | \eta^*) \\
  %   & = \sum_{i=1}^r F(\ell_i, \eta^*),
  %   \end{align*}
where $\eta^*$ is the optimal shrinker 
for the losses $L^{F,1}$ or $L^{st}$ in Table \ref{Table:formulas}.
%  (\ref{eq:squaredErrorOpt}) and
% (\ref{eq:SteinOpt}).
% given explicitly above.
\end{thm}

In short, the shrinker $\eta^*()$, which has been designed to minimize the {\em
limiting} loss, asymptotically delivers the same performance as the oracle
procedure, which has the lowest possible loss, in finite-$n$, over the entire
class of covariance estimators by arbitrary high-dimensional shrinkage rules.
On the other hand, by definition, the oracle procedure outperforms every
orthogonally-equivariant statistical estimator. We conclude that $\eta^*$  -- as
one such orthogonally-invariant estimator -- is indeed optimal (in the sense of
having the lowest limiting loss) among all  orthogonally invariant procedures.
While we only discuss the cases $L^{F,1}$ and $L^{st}$,  we suspect  that this
theorem holds true for many of the 26 loss functions considered.

%\newpage
\begin{proof}
We first outline the approach.
We can write $\Sigma$ and $\Sigma^{-1}$ in the form $I+F$, 
and $\hat \Sigma_\Delta = I + \tilde \Delta$ with
\begin{equation*}
  F = \sum_{k=1}^r \beta_k u_k u_k', \qquad \quad 
  \tilde \Delta = \sum_{i=1}^p \tilde \Delta_i v_i v_i',
\end{equation*}
where $\beta_k = \ell_k -1$ for $L^{F,1}$ and $\ell_k^{-1} - 1$ for
$L^{St}$ and $\tilde \Delta_i = \Delta_i -1$.  
Write
\begin{equation}
  \label{eq:trFdel}
  \tr \, F \tilde \Delta = \sum_{i=1}^p \tilde \Delta_i b_i, \qquad 
   b_i := \sum_{k=1}^r \beta_k (u_k' v_i)^2.
\end{equation}

For both $L = L^{F,1}$ and $L^{st}$, we establish a decomposition
\begin{equation}
  \label{eq:L-decomp}
  L_p(\Sigma,\hat \Sigma_\Delta) 
    = \sum_{i=1}^r F(\ell_i, \Delta_i) + a(\Delta_i -1)\epsilon_i 
      + \sum_{i=r+1}^p H(b_i,\Delta_i).
\end{equation}
Here, 
$a$ is a constant depending only on the loss function, 
\begin{equation}
  \label{eq:epsdef}
  \epsilon_i = b_i - \beta_i c(\ell_i)^2,
\end{equation}
and
\begin{equation} \label{eq:Hdef}
  H(b, \Delta) =
  \begin{cases}
    (\Delta-1)^2 - 2 (\Delta-1) b  \quad &  \text{for } L^{F,1} \\
    (\Delta-1)(1+b) - \log \Delta  &  \text{for } L^{St}.
  \end{cases}
\end{equation}

Decomposition (\ref{eq:L-decomp}) shows that the oracle estimator
(\ref{eq:oracle}) may be found term by term, using just univariate
minimization over each $\Delta_i$.
Consider the first sum in (\ref{eq:L-decomp}), and let $\tilde
F(\ell_i,\Delta_i)$ denote the summand. We will show that
\begin{equation}
  \label{eq:min1}
  \min_{\Delta_i} \tilde F(\ell_i,\Delta_i) 
  \ \stackrel{P}{\to} \ 
  \min_{\Delta_i}  F(\ell_i,\Delta_i),
\end{equation}
and that
\begin{equation}
  \label{eq:min2}
  \sum_{i=r+1}^p \min_{\Delta_i} H(b_i,\Delta_i)
%   = \sum_{i=r+1}^p h(b_i) 
   = O_P \left( \frac{\log^2 p}{p} \right).
\end{equation}
Together (\ref{eq:min1}) and (\ref{eq:min2}) establish the Theorem.

Turning to the details, we begin by showing (\ref{eq:L-decomp}).
For Frobenius loss, we have from our definitions and (\ref{eq:trFdel})
that
\begin{equation*}
  \| \hat \Sigma_\Delta - \Sigma \|_F^2
   = \tr (\tilde \Delta - F)(\tilde \Delta- F)'
   = \sum_{i=1}^p (\Delta_i -1)^2 - 2 (\Delta_i -1)b_i
     + \sum_{i=1}^r (\ell_i-1)^2.
\end{equation*}
For $i \geq r+1$, each summand in the first sum equals
$H(b_i,\Delta_i)$ and for $i \leq r$, we use the decomposition 
$b_i = (\ell_i-1) c(\ell_i)^2 + \epsilon_i$. We obtain decomposition
(\ref{eq:L-decomp}) with $a = -2$ and 
\begin{equation*}
  F(\ell,\Delta) = (\ell -1)^2 - 2(\ell -1)(\Delta-1) c^2 + (\Delta-1)^2.
\end{equation*}

For Stein's loss, our definitions yield
\begin{align*}
  L^{St}(\Sigma, \hat \Sigma_\Delta)
  & = \tr \tilde \Delta + \tr F + \tr F \tilde \Delta
% \tr (\Sigma^{-1} \hat \Sigma_\Delta -I) 
      - \log(|\hat \Sigma_\Delta|/|\Sigma|) \\
  % & = \tr \tilde \Delta + \tr F + \tr F \tilde \Delta 
  %    - \sum_{i=1}^p \log \Delta_i + \sum_{k=1}^r \log \ell_k \\
  & = \sum_{i=1}^p \tilde \Delta_i (1+b_i) - \log \Delta_i 
     + \sum_{k=1}^r \beta_k + \log \ell_k.
\end{align*}
Again, for each $i \geq r+1$, each summand in the first sum equals
$H(b_i,\Delta_i)$ 
and with  
$b_i = (\ell_i-1) c(\ell_i)^2 + \epsilon_i$ we obtain
(\ref{eq:L-decomp}) with  $a=1$ and
%$F(\ell,\Delta)$ as in (\ref{eq:Fdef}). 
\begin{equation*}
  F(\ell, \Delta) = (\ell^{-1}-1) + (\Delta-1)(c^2/\ell + s^2) - 
     \log(\Delta/\ell).
\end{equation*}

It remains to verify (\ref{eq:min1}) and (\ref{eq:min2}). 
Theorem 1 says that for $1 \leq i \leq r$,
\begin{equation*}
  \epsilon_i = \sum_{k = 1}^r \beta_k [(u_k'v_i)^2 - \delta_{k,i} c(\ell_i)^2]
    \stackrel{P}{\to} 0,
\end{equation*}
which yields (\ref{eq:min1}). 
From (\ref{eq:Hdef}), we observe that in our two cases
\begin{equation}
  \label{eq:quad-bd}
  h(b) := \min_\Delta H(b,\Delta) 
      =
      \begin{cases}
        -b^2 \\ 
        -b + \log(1+b)
      \end{cases}
 = O(b^2),
\end{equation}
Now, using (\ref{eq:trFdel}) and (\ref{eq:maxubd}), we get
\begin{equation*}
  \max_{r+1\leq i \leq p} |b_i| 
    \leq r \max_{1 \leq k \leq r} |\beta_k| \cdot
         \max_{i>r, k\leq r} (u_k'v_i)^2 
    = O_P \left( \frac{\log p}{p} \right).
\end{equation*}
From the previous two displays, we conclude
\begin{equation*}
%  \label{eq:min2}
  \sum_{i=r+1}^p \min_{\Delta_i} H(b_i,\Delta_i)
   = \sum_{i=r+1}^p h(b_i) 
   = O_P \left( \frac{\log^2 p}{p} \right).
\end{equation*}
which is (\ref{eq:min2}), and so completes the full proof.
\end{proof}

\section{Optimal Shrinkage with common variance $\sigma^2\neq 1$} \label{sigma:sec}

Simply put, the Spiked Covariance Model is a proportional growth
independent-variable Gaussian model,
where all variables, except the first $r$, have common variance $\sigma$. 
Literature on the spiked model often simplifies the situation by assuming
$\sigma^2=1$, as we have done in our assumption \Spike\ above.
To consider optimal shrinkage in the case of general common variance $\sigma^2>0$,
assumption \Spike\, has to be replaced by
\begin{description}
  \item[{\sc [Spike($\ell_1,\ldots,\ell_r|\sigma^2$)]}] 
    The population eigenvalues in the $n$-th problem, namely the eigenvalues of
    $\Sigma_{p_n}$, are given
    by
$(\ell_1,\ldots,\ell_r,\sigma^2,\ldots,\sigma^2)$,
where the number of ``spikes'' $r$ and
their amplitudes $\ell_1> \ldots > \ell_r\geq
1$ are fixed independently of $n$ and $p_n$.
\end{description} 
In this section we show how to use an optimal shrinker, designed for the spiked
model with common variance $\sigma^2=1$, in order to construct an optimal
shrinker for a general common
variance $\sigma^2$, namely, under assumptions \Asy\, and \SpikeSigma.

\subsection{$\sigma^2$ known}
Let $\Sigma_p$ and $S_{n,p}$ be population and sample covariance matrices,
respectively, under assumption \SpikeSigma.  When the value of $\sigma$ is known, the matrices 
$\tilde{\Sigma}_p=\Sigma_p/\sigma^2$ and the sample covariance matrix
$\tilde{S}_{n,p}=S_{n,p}/\sigma^2$ constitute population and sample covariance
matrices, respectively,  under assumption \Spike. Let $L$ be any of the loss
families considered above and let $\eta$ be a shrinker.
Define the shrinker $\tilde{\eta}$ corresponding to $\eta$ by
\begin{eqnarray} \label{translate:eq}
  \tilde{\eta}:\lambda\mapsto \sigma^2 \cdot \eta(\lambda/\sigma^2)\,.
\end{eqnarray}
Observe that for each of the loss families we consider, 
$L_p(\sigma^2 A, \sigma^2 B) = \sigma^{2 \kappa} L_p(A,B)$, where 
$\kappa \in \{ -2,-1,0,1,2\}$ depends on the family $\{L_p\}$ alone.
Hence
\[
  L_p\left(\Sigma_{p},\hat{\Sigma}_{\tilde{\eta}}(S_{n,p})\right) =
  \sigma^{2 \kappa}
  L_p\left(\tilde{\Sigma}_{p},\hat{\Sigma}_{\eta}(\tilde{S}_{n,p})\right)  
\]
%for some constant $C$ that depends on $\sigma^2$ and on the family
%$L$ alone. 
It follows that if
$\eta^*$ is the optimal shrinker for the loss family $L$, in the sense of
Definition  \ref{def:Optimal}, under Assumption \Spike\,, then 
 $\tilde{\eta}^*$ is the optimal shrinker for $L$ under Assumption 
 \SpikeSigma. Formula \eqref{translate:eq} therefore allows us to translate each
 of the optimal shrinkers given above to a corresponding optimal shrinker in the
 case of a general common variance $\sigma^2>0$. 
%assumption {\bf Spike}$(\ell_1/\sigma^2,\ldots,\ell_r/\sigma^2)$. Moreover, all
%the loss functions we considered are homogeneous w.r.t $\sigma$. It 
%follows from these two facts that
%if $\lambda \eta^*$ is an optimal shrinker derived above for the case $\sigma=1$, then 
%\lambda\mapsto $\sigma^2 \eta^*\left(  \right)<++>
%and one readily
%verifies that our optimal shrinkers, appropriately scaled by $\sigma$, 

\subsection{$\sigma^2$ unknown}

In practice, even if one is willing to assume a common variance $\sigma^2$ and
subscribe to the spiked model, the value of $\sigma^2$ is usually unknown. 
Assume however that we have a sequence of estimators
$\{\hat{\sigma}_n\}_{n=1,2,\ldots}$, where
for each $n$, $\hat{\sigma}_n$ is a real function of a $p_n$-by-$p_n$ positive
definite symmetric matrix argument. Assume further that under the spiked model
with general common variance $\sigma^2$, namely under assumptions \Asy\, and
\SpikeSigma, the sequence of estimators is consistent in the sense that 
$\hat{\sigma}_n(S_{n,p_n})\to \sigma$, almost surely. For a continuous shrinker $\eta$, 
define a sequence of
shrinkers $\left\{ \tilde{\eta}_n \right\}_{n=1,2,\ldots}$ by
\begin{eqnarray} \label{translate_asy:eq}
  \tilde{\eta}_n:\lambda\mapsto \hat{\sigma}_n^2 \cdot
  \eta\left(\lambda/\hat{\sigma}_n^2\right)\,.
\end{eqnarray}

Again for each of the loss families we consider, almost surely,
\[
  \lim_{n\to\infty}
  L_{p_n}\left(\Sigma_{p_n},\hat{\Sigma}_{\tilde{\eta}_n}(S_{n,p_n})\right)
  = \sigma^{2 \kappa}  \lim_{n\to\infty}
  L_{p_n}\left(\tilde{\Sigma}_{p_n},\hat{\Sigma}_{\eta}(\tilde{S}_{n,p_n})\right). 
\]
%where again $C$ is a constant that depends on $L$ and $\sigma^2$ alone. 
We conclude that, using \eqref{translate_asy:eq}, any consistent sequence of
estimators $\hat{\sigma}_n$ yields a sequence of shrinkers with the same
asymptotic loss as the optimal shrinker for known $\sigma^2$. In other words, at
least inasmuch as the asymptotic loss is concerned, under the spiked model, 
there is no penalty for not knowing $\sigma^2$.

Estimation of $\sigma^2$ under Assumption \SpikeSigma\, has been considered in
\cite{Kritchman2009,passemier2013,Shabalin2013} where several approaches have
been proposed.
As an simple example of a consistent sequence of estimators $\hat{\sigma}_n$, we consider
the following simple and robust approach based on matching of medians
\cite{2013arXiv1305.5870D}. The underlying idea is that for a
given value of $n$ the
sample eignevalues $\lambda_{r+1},\ldots,\lambda_{p_n}$ form an approximate
Mar\v{c}enko-Paster bulk inflated by $\sigma^2$, and that a median sample eigenvalue is
well suited to detect this inflation as it is unaffected by the sample spikes
$\lambda_1,\ldots,\lambda_r$.

Define, for a symmetric $p$-by-$p$ positive definite matrix $S$ with eigenvalues
$\lambda_1,\ldots,\lambda_p$ the quantity
\begin{eqnarray}
  \mu(S) = \frac{\lambda_{med}}{\mu_\gamma}\,,
\end{eqnarray}
where $\lambda_{med}$ is a median of $\lambda_1,\ldots,\lambda_p$ 
 and $\mu_\gamma$ is the
median of the Mar\v{c}enko-Pastur distribution, namely, the unique solution 
in $\lambda_-(\gamma)\leq x\leq \lambda_+(\gamma)$
to the equation
\begin{eqnarray*}
  \intop_{\lambda_-(\gamma)}^x
  \frac{\sqrt{(\lambda_+(\gamma)-t)(t-\lambda_-(\gamma))}}{2\pi \gamma t}dt = 
  \frac{1}{2}\,,
\end{eqnarray*}
where as before $\lambda_\pm(\gamma) = (1\pm\sqrt{\gamma})^2$.
Note that the median $\mu_\gamma$ is not available analytically but
can easily be obtained 
numerically, for example using remarks on the Mar\v{c}enko-Pastur
cumulative distribution function included in SI.
%by numerical quadrature. 
Now for a sequence $\left\{ S_{n,p_n} \right\}$ of
sample covariance matrices, define the sequence of estimators 
\begin{eqnarray} \label{sigmahat:eq}
  \hat{\sigma}_n:S_{n,p_n} \mapsto \sqrt{\mu(S_{n,p_n})}\,.
\end{eqnarray}

\begin{lemma} \label{sigmahat:lem}
  Let $\sigma^2>0$, and assume \Asy\, and \SpikeSigma. Then almost surely
  \[\lim_{n\to\infty} \hat{\sigma}_n(S_{n,p_n}) = \sigma .\]
\end{lemma}

In summary, using \eqref{translate:eq} (for $\sigma^2$ known) or
\eqref{translate_asy:eq} with \eqref{sigmahat:eq} (for $\sigma^2$ unknown) one
can use the optimal shrinkers for
each of the loss families discussed above, designed for the case $\sigma=1$, to
construct a
shrinker that is optimal, for the same loss family, under the spiked model with
common variance $\sigma^2\neq 1$.

\section{Discussion} \label{discussion:sec}

In this paper, we considered covariance estimation in high
dimensions, where the dimension $p$ is comparable to the number of observations
$n$. We chose a fixed-rank principal subspace, and let the dimension of the
problem grow large.  A different asymptotic framework for covariance estimation
would choose a principal subspace whose rank is a fixed {\em fraction} of the problem
dimension; i.e. the rank of the principal subspace is growing rather than fixed.
(In the sibling problem of matrix denoising, compare the ``spiked'' setup
\cite{2013arXiv1305.5870D,Shabalin2013,GavishDonohoSingShrink} with the ``fixed fraction'' setup  of
\cite{2013arXiv1304.2085D}.)

In the fixed fraction framework, some of underlying phenomena remain qualitatively similar to
those governing the spiked model, while new effects appear.
Importantly, the
relationships used in this paper, predicting the location of the top empirical
eigenvalues, as well as the displacement of empirical eigenvectors,
in terms of the top theoretical eigenvalues, no longer hold.
Instead, a complex nonlinear relation exists between the limiting
distribution of the empirical eigenvalues and the limiting distribution of the
theoretical eigenvalues, as expressed by the Mar\v{c}enko-Pastur (MP) relation
between their Stieltjes transforms
\cite{marvcenko1967distribution,bai2010spectral}.

Covariance shrinkage in the proportional rank model should then, naturally, make
use of the so-called {\em MP Equation}.  Noureddine El Karoui
\cite{karoui2008spectrum}  proposed a method for debiasing the empirical
eigenvalues, namely, for estimating (in a certain specific sense) their corresponding
population eigenvalues; Olivier Ledoit and Sandrine Pech\'{e}
\cite{ledoitPeche} developed analytic tools to also account for the inaccuracy
of empirical eigenvectors,  and Ledoit and Michael Wolf
\cite{ledoit2012nonlinear} have implemented such tools and applied them in this
setting.

% We admire the depth of insight involved in the study of the proportional rank
% case, which is really quite subtle and beautiful.  
The proportional rank
case is indeed subtle and beautiful.  
Yet, the fixed-rank case
deserves to be worked out carefully. In particular, the shrinkers we have
obtained here in the fixed-rank case are extremely simple to implement,
requiring just a few code lines in any scientific computing language. In
comparison, the covariance estimation ideas of
\cite{karoui2008spectrum,ledoit2012nonlinear}, based on powerful and deep
insights from MP theory, require a delicate, nontrivial effort to implement in software,
and call for expertise in numerical analysis and optimization.  As a
result, the simple shrinkage rules we propose here 
may be 
%are 
more likely to be applied correctly
in practice, and to work as expected, even in relatively small sample sizes.

An analogy 
%of this situation 
can be made to shrinkage in the normal means
problem, for example \cite{donoho1994minimax}.  In that
problem, often a full Bayesian model applies, and in principle a Bayesian
shrinkage would provide an optimal result \cite{brown2009}. Yet, in applications
one often wants a simple method which is easy to implement correctly, and which
is able to deliver much of the benefit of the full Bayesian approach.  In
literally thousands of cases, simple methods of shrinkage \-- such as
thresholding \-- have been chosen over the full Bayesian method for precisely
that reason.

\section*{Reproducible Research}
\label{ssec:reproResearch}

In the code supplement
  \cite{SDR} we offer 
a Matlab software library that includes:
\begin{enumerate}
  \item A function to compute the value of each of the 26 optimal shrinkers
    discussed to high precision.
  \item A function to test the correctness of each of the 18 analytic shrinker fomulas
    provided.
  \item Scripts that generate each of the figures in this paper, or subsets of
    them for specified loss functions. 
\end{enumerate}

%%% Local Variables:
%%% mode: latex
%%% TeX-master: "build_main"
%%% End:

%\input{body_SI}

\section*{Acknowledgements}

We thank Amit Singer, Andrea Montanari, Sourav Chatterjee and Boaz Nadler for
helpful discussions. We also thank the anonymous referees for significantly
improving the manuscript through their helpful comments.  
This work was partially supported by NSF DMS-0906812  (ARRA).
MG was partially supported by a
William R. and Sara Hart Kimball Stanford Graduate Fellowship.

\section*{Proofs and Additional Results}
%\sdatatype{.pdf}
%\sdescription{
In the supplementary material \cite{SI} we provide proofs omitted from the main text for
  space considerations and auxiliary lemmas used in various proofs.
   Notably, we prove 
   Lemma \ref{bulk:lem},
  %\ref{lem:operatorNormNonlinearity},
   %\ref{lem:frobeniusNormNonlinearity},
  %\ref{lem:nuclearNormNonlinearity}, 
   %\ref{lem:steinLossNonlinearity} and
  %\ref{lem:frechetLossNonlinearity},    
  and provide detailed derivations of the
  17 explicit formulas for optimal shrinkers, as summarized in Table 
  \ref{Table:formulas}.
 In addition, in the supplementary material we offer a detailed study of the 
   large-$\lambda$ asymptotics (asymptotic slope and asymptotic shift) of the
   optimal shrinkers discovered in this paper, and tabulate the asymptotic
   behavior of each optimal shrinker.  We also study the asymptotic percent
   improvement of the optimal shrinkers over naive hard thresholding of the
   sample covariance eigenvalues.
 %}

%\input{citations}
% TODO: replace this
\bibliography{SpikedCovar,temp}
\bibliographystyle{unsrt}

% this needs \usepackage{xr,refcount} in sourcing doc's preamble

\makeatletter
\newcommand*{\storecounter}[2]{%
  \edef\@currentlabel{\csname the#1\endcsname}% Store current counter value in \@currentlabel
  \label{#2}% Store label
}
\makeatother

\storecounter{thm}{extthm}
\storecounter{cor}{extcor}
\storecounter{defn}{extdefn}
\storecounter{lemma}{extlemma}
\storecounter{prop}{extprop}
\storecounter{figure}{extfigure}

\end{document}